\documentclass[12pt]{article}
\usepackage{latexsym,eucal,amsmath,amssymb,color}
\usepackage{setspace}
\usepackage{tikz}
\usepackage{graphicx,subcaption,amsfonts}
\usepackage{amsthm}
\usepackage{enumerate}

\newcommand{\MF}{\Upsilon}
\newcommand{\AC}{A}   

\numberwithin{equation}{section}

\newtheorem{Theorem}{Theorem}[section]
\newtheorem{Lemma}[Theorem]{Lemma}
\newtheorem{Definition}[Theorem]{Definition}
\newtheorem{Question}[Theorem]{Question}
\newtheorem{Corollary}[Theorem]{Corollary}
\newenvironment{Proof}{\begin{trivlist} \item[] {\bf Proof.}}{\hfill $\Box$\end{trivlist}}

\newcommand{\NP}{$\mathsf{NP}$}
\newcommand{\sharpP}{$\mathsf{\#P}$}

\DeclareMathAlphabet{\ORIGmathcal}{OMS}{cmsy}{m}{n}   

\newcommand{\Complex}{\mathbb C}
\newcommand{\QED}{\hfill $\Box$}

\newcommand{\ds}{\displaystyle}

\renewcommand{\geq}{\geqslant}
\renewcommand{\leq}{\leqslant}

\title{Acyclic polynomials of graphs}

\author{Caroline Barton and Jason I.\ Brown \\
carolinebarton@dal.ca, jason.brown@dal.ca \\
Department of Mathematics and Statistics\\
Dalhousie University \\
Halifax, Nova Scotia \\
Canada. B3H 4R2 \\ \\
David A.\ Pike\\
dapike@mun.ca\\
Department of Mathematics and Statistics\\
Memorial University of Newfoundland\\
St.\ John's, Newfoundland\\
Canada. A1C 5S7 }

\begin{document}

\date{\today}

\maketitle

\begin{center} \large Abstract \end{center}
For each nonnegative integer $i$, let $a_i$ be the number of $i$-subsets of $V(G)$ that induce an acyclic subgraph of a given graph $G$.
We define $\AC(G,x) = \sum_{i \geq 0} a_i x^i$ (the generating function for $a_i$) to be the {\it acyclic polynomial} for $G$.
After presenting some properties of these polynomials, we investigate the nature and location of their roots.

\vspace*{\baselineskip}

\vspace*{\baselineskip}


\section{Introduction}

A variety of graph polynomials have been studied, both for applied and theoretical considerations. Perhaps the best known family, {\em chromatic polynomials}, counts the number of proper colourings of graphs, and was introduced in the study of the Four Colour Problem, but has morphed over the years into an important field in its own right (see, for example, \cite{dongbook}). {\em Reliability polynomials} are a well-studied model of network robustness to probabilistic failures, and have attracted interest for both applied and pure perspectives (see \cite{colbook}). Other graph polynomials have arisen as generating functions for subsets of vertex sets or edge sets of a graph, especially those having certain properties. Such polynomials allow for a finer investigation of the graph property in question, allowing an encoding of the minimum or minimum cardinality of such sets and the totality of the number of such sets.
For example, {\em independence polynomials} are generating functions for independent sets of a graph, while {\em domination polynomials} enumerate dominating sets. For all these graph polynomials, work has varied from calculation and optimality to analytic properties and roots.

Given a (finite, undirected) graph $G$, we define the  {\it acyclic polynomial}, $\AC(G,x)$, of $G$ to be the generating function
for the number of acyclic subsets of $V(G)$ ({\it i.e.}, vertex subsets that induce acyclic subgraphs of $G$).
Specifically,
$$\ds \AC(G,x) = \sum_{i \geq 0} a_i x^i$$
where
$a_i = |\{ S \subseteq V(G) : |S|=i \mbox{~and~} G[S] \mbox{~is acyclic} \}|$
is the number of acyclic vertex sets of cardinality $i$.%
\footnote{Note that the term ``acyclic polynomial'' already exists within the scientific literature.
Historically, it referred to a polynomial
that is now more commonly described as the ``matching polynomial'' and which is closely related to the generating function for the number of
matchings of size $i$ within a graph.  See~\cite[page 263]{EMM2011} for more details.
Given that the term ``acyclic polynomial'' has fallen out of use from its original context, we now put the term to new use
by defining it as the generating function
for the number of acyclic induced subgraphs of order $i$ that are within a given graph.
}
We remark that the collection $\mathcal{A}(G)$ of acyclic vertex subsets (that is, subsets of $V(G)$ that induce an acyclic subgraph)  of a graph $G$ forms a {\it (simplicial) complex},
that is, it is closed under containment, and the acyclic polynomial of $G$ is what is known as the {\it face polynomial}, or simply $f$-{\it polynomial},
of the complex (the {\it (combinatorial) dimension} of the complex is the maximum size of any set in the complex, and we shall say that $G$ has {\em  acyclic dimension} $d$ if the acyclic complex of $G$,  $\mathcal{A}(G)$, has dimension $d$).

In the remainder of this first section of the paper we discuss several aspects acyclic polynomials
such as how they relate to decycling of graphs,
how they do (or do not) encode certain graph invariants,
the computational complexity of determining and/or evaluating them,
and showing that they do not arise from evaluations of Tutte polynomials.

In Section~\ref{Section-AcyclicRoots} we focus on {\em acyclic roots}, namely the roots of acyclic polynomials.
For acyclic polynomials of degree $3$ we show that the roots all lie in the left half of the complex plane, and that arbitrarily large modulus is possible.
For acyclic polynomials more generally ({\it i.e.}, of any degree) we characterize those graphs for which the roots are all real
as well as which rational numbers are acyclic roots.
We also consider the maximum and minimum growth of the moduli of the roots, and we show that roots can exist in the right half-plane.
In Section~\ref{Section-OpenProblems} we conclude with several open problems.

\subsection{Acyclic Polynomials and Decycling}

Any subset $S$ of the vertex set $V(G)$ of a graph $G$ such that $G-S$ (the subgraph of $G$ that is induced by the vertices of $V(G) \setminus S$) is a forest
({\it i.e.}, acyclic) is known as a {\it decycling set} or {\it feedback vertex set} for the graph $G$.
For an introduction to the topic of decycling of graphs, see~\cite{BB2002,BV1997}.
Determining whether an arbitrary graph $G$ has a decycling set of a given cardinality is an {\NP}-complete problem~\cite{Karp1972}.
Nevertheless, calculating the size $\nabla(G)$ of a smallest decycling set for a graph $G$
is a problem of practical interest as it has several natural applications,
such as that of avoiding short-circuits and other forms of feedback in electrical networks~\cite{FPR1999}.
For certain classes of graphs this problem is tractable, such as for complete graphs, complete bipartite graphs,
cubic graphs~\cite{LL1999,UKG1988}, Cartesian products of cycles~\cite{PikeZou2005}, and generalized Petersen graphs~\cite{GXWZY2015}.
For instance,
$\nabla(K_n) = n-2$ whenever $n \geq 2$,
$\nabla(K_{m,n}) = \min\{m,n\}-1$ whenever $\min \{m,n\} \geq 2$,
and $\nabla(C_m \Box C_n) = \lceil \frac{mn+2}{3} \rceil$ whenever $m,n \in \{3,5,6,7,\ldots\}$.

The complement $V(G) \setminus S$ of a decycling set $S$ in a graph $G$ induces an acyclic subgraph of $G$.
Let $\MF(G) = |V(G)| - \nabla(G)$ denote the {\em acyclic dimension} of $G$ and observe that $\MF(G)$ is also the degree of the acyclic polynomial $\AC(G,x)$.
Hence determining $\AC(G,x)$ enables the decycling number $\nabla(G)$ to be found.
This observation, and the potential that acyclic polynomials might address some problems involving decycling of graphs,
provided our initial motivation for studying acyclic polynomials.

We note that the acyclic dimension of graphs has itself been an active area of research.
At a conference in 1977, Albertson and Berman posed the still-open conjecture that $\MF(G) \geq |V(G)|/2$ for any simple planar graph $G$~\cite{AB1979};
a result of Borodin implies that $\MF(G) \geq 2 |V(G)| / 5$ for every planar graph $G$~\cite{Borodin1979}.
For graphs in general (not necessarily planar) in 1987, Alon, Kahn and Seymour
established a lower bound of
$\MF(G) \geq \sum_{v \in V(G)} \min\{1, \frac{2}{1 + \deg(v)} \}$
and hence $\MF(G) \geq 2|V(G)| / (1+ \Delta(G))$, where $\Delta(G)$ denotes the maximum degree of $G$~\cite{AKS1987}.
More recently it has been shown that
$\MF(G) \geq (8|V(G)| - 2|E(G)| -2)/9$
for every connected graph $G$~\cite{ShiXu2017}
and also that
$\MF(G) \geq 6|V(G)| / (2\Delta(G) + \omega(G) +2)$ where $\omega(G)$ denotes the order ({\it i.e.}, the number of vertices) of a maximum clique in $G$~\cite{Kogan}.

\subsection{Acyclic Polynomials and Graph Invariants}

While determining the degree $\MF(G)$ of $\AC(G,x)$ is generally intractable,
certainly some individual coefficients of the acyclic polynomial $\AC(G,x) = \sum_{i \geq 0} a_i x^i$ of a given graph $G$ can be easily be determined.
Let $\sigma_i = \sigma_i(G)$ denote the number of $i$-cycles in a graph $G$ of order $n$ and observe:

\begin{itemize}

\item Clearly $a_i \leq \binom{n}{i}$ for all $i$.
Moreover, if a non-acyclic graph $G$ has girth $g$ then
$a_i = \binom{n}{i}$ for each $i \in \{0,1,\ldots, g-1\}$
and
$a_{g}  = \binom{n}{g} - \sigma_g < \binom{n}{g}$,
whereas if $G$ is acyclic then
$a_i = \binom{n}{i}$ for each $i \leq n$.
In particular, $a_0 = 1$, $a_1 = n$ and $a_2 = \binom{n}{2}$.

\item As $a_i$ is the number of sets (or {\it faces}) of cardinality $i$ in the complex $\mathcal{A}(G)$ and the complex clearly has dimension  $\MF(G) = n-\nabla(G)$, we have  $a_i > 0$ for each $i \in \{0,1,\ldots,\MF(G)\}$ and $a_i = 0$ for each $i > \MF(G)$.

\item There are well known inequalities, known as {\it Sperner bounds} (see, for example, \cite{Sperner1928}), for the cardinalities of  faces of each size in a complex, and these imply that
$$a_{i} \leq \left( \frac{n-i+1}{i} \right) a_{i-1}.$$

\end{itemize}

We summarize these observations as follows:

\begin{Theorem}
If $G$ and $H$ have the same acyclic polynomial, then they have the same order, girth and decycling number. \QED
\end{Theorem}

While the acyclic polynomial encodes the order, girth and decycling number of a graph, it does not encode some other basic invariants:

\begin{itemize}

\item The acyclic polynomial does not encode the number of edges. For example, any two acyclic graphs of order $n$ share the same acyclic polynomial $(1+x)^n$, but they can have a different number of edges if $n \geq 2$. Even for connected graphs that contain cycles, there are examples.
Let $n$ be an odd integer,
let $G_1$ denote the graph obtained by
adding a single pendant vertex to one of the vertices of $K_{n-1}$,
and let $G_2$ denote the graph obtained by
removing the $\frac{n-1}{2}$ edges of a maximum matching from the complete graph $K_n$.
Then $\AC(G_1,x) = \AC(G_2,x) = 1 + nx + \binom{n}{2}x^2 + \binom{n-1}{2}x^3$.
However, for all $n \geq 5$, $|E(G_1)| \neq |E(G_2)|$
and so as $n$ varies we obtain an infinite family of pairs of distinct connected graphs that share the same acyclic polynomial
but have different numbers of edges.

\item The acyclic polynomial does not encode whether a graph is bipartite.
To demonstrate this, it suffices to find two graphs, one bipartite and the other not bipartite, that share the same acyclic polynomial.
Two such graphs are illustrated in Figure~\ref{Fig:BipartiteEncode}.
These two graphs both have
$1 + 7x + 21x^2 + 35x^3 + 32x^4 + 12x^5$ as their acyclic polynomial.

\item Whereas the acyclic polynomial does encode the girth of a graph, it can be observed from the graphs shown in
Figure~\ref{Fig:BipartiteEncode}
that the circumference of a graph is not encoded
(note that the graph on the right is Hamiltonian, but the one on the left is not Hamiltonian).

\end{itemize}

\begin{Theorem}
The acyclic polynomial does \underline{not} encode the number of edges, the bipartiteness or the circumference of a graph. \QED
\end{Theorem}

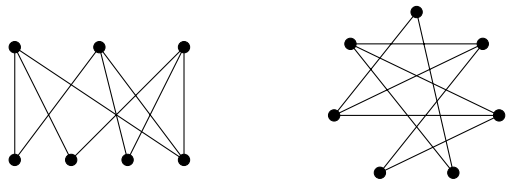
\begin{figure}[htbp]
\begin{center}

\begin{tikzpicture}[scale=.75]
\draw [-, black] (1.5,2) -- (0,0) -- (0,2) -- (1,0) -- (3,2) -- (3,0) -- (1.5,2) -- (2,0) -- (3,2);
\draw [-, black] (0,2) -- (3,0);

\draw [black,fill] (0,0) circle [radius=0.1];
\draw [black,fill] (1,0) circle [radius=0.1];
\draw [black,fill] (2,0) circle [radius=0.1];
\draw [black,fill] (3,0) circle [radius=0.1];

\draw [black,fill] (0,2) circle [radius=0.1];
\draw [black,fill] (1.5,2) circle [radius=0.1];
\draw [black,fill] (3,2) circle [radius=0.1];

\draw [white,fill] (0,-0.4) circle [radius=0.001];

\end{tikzpicture}
\hspace*{15mm}
\begin{tikzpicture}[scale=.75]

\draw [black,fill](                            0., 1.5) circle [radius=0.1] node(A) {};
\draw [black,fill](                   -1.172747224, 0.9352347030) circle [radius=0.1] node(G) {};
\draw [black,fill](                  -1.462391868, -0.3337814010) circle [radius=0.1] node(F) {};
\draw [black,fill](                  -0.6508256091, -1.351453302) circle [radius=0.1] node(E) {};
\draw [black,fill](                   0.6508256091, -1.351453302) circle [radius=0.1] node(D) {};
\draw [black,fill](                   1.462391868, -0.3337814010) circle [radius=0.1] node(C) {};
\draw [black,fill](                   1.172747224, 0.9352347030) circle [radius=0.1] node(B) {};

\draw [-, black] (A.center) -- (D.center) -- (G.center) -- (B.center) -- (F.center) -- (A.center);
\draw [-, black] (B.center) -- (E.center) -- (C.center) -- (F.center);
\draw [-, black] (G.center) -- (C.center);

\end{tikzpicture}

\caption{A bipartite graph and a non-bipartite graph with identical acyclic polynomials}
\label{Fig:BipartiteEncode}
\end{center}
\end{figure}

\subsection{The Complexity of Calculating Acyclic Polynomials}

For some families of graphs, the entire acyclic polynomial can be determined fairly easily.
For instance, for complete graphs we have $\AC(K_n,x) = 1 + nx + \binom{n}{2}x^2$
and for cycles we have $\AC(C_n,x) = \sum_{i=0}^{n-1} \binom{n}{i}x^i = (1+x)^n - x^n$.
As a more interesting example, {\it cographs} are those graphs that can be built recursively from a single vertex via disjoint union and join operations.
Equivalently, cographs can be characterized as those that do not contain any path $P_4$ of order 4 as an induced subgraph~\cite{CLSB1981}.
Given two graphs $G$ and $H$, we denote their disjoint union as $G \cup H$
and their join as $G + H$
(the {\it join} of two graphs $G$ and $H$ is formed from their disjoint union by adding in all edges between a vertex of $G$ and a vertex of $H$).
It is elementary that
\begin{equation}\label{Eq0}
\AC(G \cup H,x) = \AC(G,x) \cdot \AC(H,x)
\end{equation}
as the union of acyclic vertex sets from disjoint graphs $G$ and $H$ is acyclic in $G \cup H$.
The effect of the join operation is more subtle, and it involves one other graph polynomial. For any graph $G$, let $I(G,x)$ be the independence polynomial for $G$, that is, the generating function of the {\it independent sets} of $G$ (a set of vertices is {\it independent} if it contains no edge). Independence polynomials have been well studied (see, for example, \cite{GutmanHarary1983,LevitMandrescu2005}). Their behaviour under disjoint  union and join is quite straightforward:
\begin{equation}
\label{Eq1}
I(G \cup H,x)=(I(G,x)) \cdot (I(H,x))
\end{equation}
and
\begin{equation}
\label{Eq2}
I(G + H,x) = I(G,x) + I(H,x) -1.
\end{equation}

We now address how to calculate the acyclic polynomial of the join of two graphs from
the acyclic and independence polynomials of each of the two graphs being joined.

\begin{Theorem}
\label{Thm-join}
For any graphs $G$ and $H$, of orders $n_G$ and $n_H$ respectively,
\begin{eqnarray*}
\AC(G + H,x) & = &
\AC(G,x) + \AC(H,x) +
n_G x I(H,x)
+
n_H x I(G,x)\\
 & &
+\left( {n \choose 2} -2 n_G n_H - {{n_G} \choose 2} - {{n_H} \choose 2}\right)x^2
- nx -1,
\end{eqnarray*}
where $n = n_G + n_H$.
\end{Theorem}

\begin{Proof}
Suppose $S \subseteq V(G + H)$ induces an acyclic subgraph of $G + H$.
Let $S_G = S \cap V(G)$ and $S_H = S \cap V(H)$.
Necessarily $\min \{|S_G|,|S_H|\} \leq 1$ for otherwise the subgraph of $G+H$ induced by $S$ is not acyclic.
Hence either
\begin{enumerate}[(i)]
\item $|S_G| = 0$ and $S_H$ induces an acyclic subgraph of $H$, or
\item $|S_G| = 1$ and $S_H$ consists of an independent set in $H$, or
\item $|S_H| = 0$ and $S_G$ induces an acyclic subgraph of $G$, or
\item $|S_H| = 1$ and $S_G$ consists of an independent set in $G$.
\end{enumerate}

By using the notation $[x^t]\, f(x)$ to denote the coefficient of the $x^t$ term of the polynomial $f(x)$,
then
\[[x^i]\, \AC(G + H,x) = [x^i]\, \AC(H,x) + |V(G)| [x^{i-1}]\, I(H,x) + [x^i]\, \AC(G,x) + |V(H)| [x^{i-1}]\, I(G,x)\]
when $i \geq 3$.  By summing over all $i \geq 3$, it follows that

\begin{eqnarray*}
\AC(G+H,x) - {n \choose 2}x^2-nx-1 & = & \left( \AC(H,x) - {{n_H} \choose 2}x^2-n_{H}x-1 \right)  \\
 & & {}+ n_G x \left( I(H,x) - n_H x - 1 \right)  \\
  & & {}+ \left( \AC(G,x) - {{n_G} \choose 2}x^2-n_{G}x-1 \right)  \\
   & & {}+ n_H x \left( I(G,x) - n_G x - 1 \right).
\end{eqnarray*}
Therefore
\begin{eqnarray*}
\AC(G+H,x) & = & \AC(G,x) + \AC(H,x) +
n_G x I(H,x)
+ n_H x I(G,x) +\\ & &
\left( {n \choose 2} -2 n_G n_H - {{n_G} \choose 2} - {{n_H} \choose 2}\right)x^2
- nx -1.
\end{eqnarray*}

%
%
%
%
%
%
%
%
\end{Proof}

%
%
%

It was observed by one of our referees that the acyclic polynomial of a graph 
can be expressed in monadic second-order logic (MSOL).
Specifically, note that 
$$ \AC(G,x) = \sum_{A \subseteq V(G) \,:\, (V(G),E(G),A) \models \varphi(A)}  x^{|A|}$$
where 
$\varphi(A)$ is an MSOL expression indicating that $A$ is an acyclic set in $G$.
Equivalently, $\varphi(A)$ indicates that the subgraph of $G$ induced by $A$ has no $K_3$-minor.
Section~1.3 of \cite{CEbook} shows how to formulate a logical expression to indicate that a graph has no $H$-minor,
where $H$ is any fixed simple loopless graph.
As a consequence of being able to express $\AC(G,x)$ in monadic second-order logic, 
it follows by a theorem due to Courcelle, Makowsky and Rotics~\cite{makowskylogic} 
that the evaluation of $\AC(G,x)$ for graphs of bounded clique-width is fixed parameter tractable.
In the case of cographs (which have clique-width at most $2$) we are able to do much better.

\begin{Theorem}
\label{Thm:Cographs}
If $G$ is a cograph, then $\AC(G,x)$ can be calculated in linear time.
\end{Theorem}

\begin{Proof}
Recall that cographs are graphs that can be constructed through a collection of disjoint union and join operations.
Recognizing that a graph is a cograph and determining the sequence of operations that comprise its construction
can be accomplished in linear time~\cite{CPS1985}.
For a cograph of order $n$, this sequence consists of $n-1$
operations.
For each disjoint union operation, the acyclic and independence polynomials of the graph arising from the operation can be
calculated from the polynomials of the two ingredient graphs by using Equations~(\ref{Eq0}) and~(\ref{Eq1}).
For each join operation, the acyclic and independence polynomials of the graph arising from the operation can be
calculated by using Equation~(\ref{Eq2}) and Theorem~\ref{Thm-join}.
For either operation these calculations take constant time.
\end{Proof}

Since it is, in general, \NP-hard to determine the decycling number $\nabla(G)$ of a graph $G$,
it is likewise intractable to determine the acyclic polynomial $\AC(G,x)$ for a general graph.
However, we can ask whether the task of evaluating the acyclic polynomial for certain choices of $x$ (without actually determining the polynomial itself)
can be performed efficiently.
This type of question has been investigated for other graph polynomials; for the chromatic polynomial see~\cite{Goodall2018,Linial1986},
and for the independence polynomial see~\cite{BrownHoshino2009}.
For us to proceed, observe that
Theorem~\ref{Thm-join} enables
the following interesting and important connection between acyclic and independence polynomials for graphs to be derived.

\begin{Corollary}
For any graph $G$, $\AC(G + K_1, x) = \AC(G,x) + x I(G,x)$.
\label{Cor:AC_AC_I}
\end{Corollary}

In particular note that an oracle for determining (or evaluating) acyclic polynomials therefore enables
the calculation (or evaluation) of independence polynomials.
It now follows that the only acyclic polynomial evaluation that is tractable is for $x = 0$,
for which $\AC(G,0)=1$ for every graph $G$.


\begin{Theorem}
Evaluating the acyclic polynomial for an arbitrary graph $G$ and nonzero $x$ is intractable.
\end{Theorem}

\begin{Proof}
It is known that evaluating $I(G,x)$ is {\sharpP}-hard for each $x \in \Complex \setminus \{0\}$~\cite{BrownHoshino2009}.
It therefore follows from Corollary~\ref{Cor:AC_AC_I} that evaluating the acyclic polynomial
for an arbitrary graph $G$ and nonzero $x$ is also intractable.
\end{Proof}

It follows immediately that determining $\AC(G,1)$, the total number of induced forests of a graph $G$, is {\sharpP}-hard.

\subsection{The Acyclic Polynomial is \underline{not} an Evaluation of the Tutte Polynomial}
\label{Sec:NotTutte}
If one concerns oneself with acyclic {\it edge} sets rather than vertex sets, then the corresponding complex is in fact a well known matroid,
the {\it graphic matroid} of graph $G$, and its $f$-polynomial is a simple evaluation of $G$'s {\it Tutte polynomial}.
The acyclic polynomials we propose here do not arise as evaluations of Tutte polynomials, 
as can be seen by the following argument 
(which was provided to us by an anonymous referee).

The \textit{most general edge elimination invariant} (the $\xi$ {\em polynomial}) was introduced in \cite{agm08,agm10} as a generalization of the Tutte and matching polynomials. In the definition below, $-e$, $/e$ and $\dagger e$ denote the edge deletion, contraction and extraction of $e$ from $G$ (the \textit{extraction} is the graph formed from $G$ by removing the endpoints of $e$ and all incident edges), and $G \cup H$ is the disjoint union of $G$ and $H$.

\begin{Definition}
Let $F$ be a graph parameter with values in a ring ${\mathcal R}$. $F$ is an EE-invariant if there exist $\alpha,\beta,\gamma \in {\mathcal R}$ such that
\[ F(G) = F(G_{-e}) + \alpha F(G_{/e}) + \beta F(G_{\dagger e}),\]
where $e \in E(G)$, with the additional conditions that $F(\emptyset)=1$, $F(K_{1}) = \gamma$ and $F(G \cup H) = F(G)\cdot F(H)$.
\end{Definition}

Let $\xi(G;x,y,z)$ be the graph polynomial defined by
\[ \xi(G;x,y,z) = \sum_{A,B \, \subseteq \, E(G)} x^{c(A \cup B) -\mbox{cov}(B)}y^{|A| + |B| +c(A \cup B) -\mbox{cov}(B)}z^{\mbox{cov}(B)},\]
where the summation is over all subsets $A$ and $B$ of $E(G)$ such that the vertex subsets $V(A)$ and $V(B)$ covered by $A$ and $B$, respectively, are disjoint, $c(A)$ is the number of connected components in $(V(G),A)$, and $\mbox{cov}(B)$ is the number of connected components of $(V(B),B)$.

\begin{Theorem}{\cite{agm08}}
Let $G$ be a graph. Then
\begin{enumerate}[(i)]
\item $\xi(G;x,y,z)$ is an EE-invariant.
\item Every EE-invariant is a substitution instance of $\xi(G;x,y,z)$ multiplied by some factor $s(G)$ which only depends on the number of vertices, edges and connected components of $G$.
\item Both the matching polynomial and the Tutte polynomial are EE-invariants given by
\[ T(G;x,y) = (x-1)^{-c(E(G))}(y-1)^{-|V(G)|}\xi(G;(x-1)(y-1),y-1,0),\]
and
\[ M(G;w_1,w_2) = \xi(G;w_1,0,w_2).\]
\end{enumerate}
\end{Theorem}

\begin{Theorem}
The acyclic polynomial $\AC(G,x)$ is not an EE-invariant, and hence not an evaluation ({\it i.e.}, substitution instance) of the Tutte polynomial (or the matching polynomial).
\end{Theorem}

\begin{proof}
Assume, to reach a contradiction, that the acyclic polynomial $\AC(G,x)$ is an EE-invariant for some $\alpha,\beta,\gamma$. For the path $P_n$ of order $n$ with $e$ an edge adjacent to a leaf, we have that
\[ (P_n)_{-e} = K_1 \cup P_{n-1}, \hspace*{3em} (P_{n})_{/e} = P_{n-1}, \hspace*{3em} (P_{n})_{\dagger e} = P_{n-2},\]
while for the cycle $C_n$ of order $n$ with $e$ any edge of the cycle,
\[ (C_n)_{-e} = P_{n}, \hspace*{3em} (C_{n})_{/e} = C_{n-1}, \hspace*{3em} (C_{n})_{\dagger e} = P_{n-2}.\]
From
\[ \AC(P_{n},x) = (1+x)^n = (1+x)(1+x)^{n-1} + \alpha(1+x)^{n-1} + \beta(1+x)^{n-2}\]
we find that $\alpha = \beta = 0$. However, then
\[ \AC(C_{n},x) = (1+x)^n - x^n = (1+x)^{n} + \alpha ((1+x)^{n-1}-x^{n-1}) + \beta(1+x)^{n-2},\]
from which it follows that $x^n = 0$, a contradiction. Thus $\AC(G,x)$ is not an EE-invariant, and hence not an evaluation of the Tutte polynomial (or the matching polynomial).
\end{proof}

\section{Acyclic Roots}
\label{Section-AcyclicRoots}

The roots of many graph polynomials have received considerable attention:
\begin{itemize}
\item For chromatic polynomials, the chromatic number of a graph $G$ is simply the least positive integer that is \underline{not} a root of its polynomial, and the infamous Four Colour Theorem can be stated as: $4$ is never a root of a chromatic polynomial of a planar graph. There are many, many results on {\em chromatic roots}, that is, the roots of chromatic polynomials (see, for example, \cite[Chapters 12-14]{dongbook}), including that chromatic roots are dense in the whole complex plane, the closure of the real chromatic roots is $[\frac{32}{27},\infty)$, and that the chromatic roots of a graph with maximum degree $\Delta$ are within the disk $\{z \in {\mathbb C}: |z| \leq 8\Delta\}$.
\item Given a graph $G$ where vertices are always operational but each edge is independently operational with probability $p$, the {\em all-terminal reliability (polynomial)} of $G$ is the probability that all vertices can communicate, that is, the operational edges contain a spanning tree of the graph. The roots of all-terminal reliability polynomials were studied first in \cite{browncolbourn}, where it was conjectured that they all lay in the unit disk centered at $z = 1$ (the real roots were shown to be in the disk, and the closure of the roots contained the disk). While the conjecture remained open for approximately 10 years, it was finally shown \cite{roylesokal} to be false, but only by the slimmest of margins (the furthest a root is known to be away from $z = 1$ is approximately $1.14$ \cite{brownlucas}). It is still unknown whether roots of all-terminal reliability are unbounded.
\item The {\em independence polynomial} $I(G,x)$ of a graph $G$ is the generating polynomial $\sum i_k x^k$ for the number of independent sets $i_k$ of each cardinality $k$. There are many interesting results about the roots of independence polynomials ({\it cf.}\ \cite{LevitMandrescu2005}), including Chudnovsky and Seymour's beautiful paper \cite{chudseymour} showing that the roots of independence polynomials of claw-free graphs ({\it i.e.}, graphs that do not contain an induced star $K_{1,3}$) are real (and hence the coefficients are {\em unimodal}, that is, nondecreasing, then nonincreasing).
\end{itemize}

More generally, the roots of polynomials are interesting for a number of reasons.
The nature and location of roots can also show relationships among the coefficients, and the regions that (or do not) contain roots can show interesting structure.
For example, if a polynomial $f$ with positive coefficients has all real roots, then the sequence of coefficients of $f$ is unimodal, and extension shows the same is true provided all of the roots of $f$ lie in the sector $\{z \in {\mathbb C}: 2\pi/3 \leq \mbox{arg} z \leq  2\pi/3\}$ (see~\cite{brenti}). Michelen and Sahasrabudhe~\cite{michelen} prove that if a polynomial $h$ is the probability generating function of a random variable $X \in \{ 0, 1, \ldots , n \}$ with sufficiently large standard deviation, and the polynomial has no roots ``close'' to $1$, then $X$ is approximately normally distributed. As yet another example, Barvinok~\cite{barvinok} built deterministic quasi-polynomial-time approximation algorithms for approximating polynomial functions using zero-free regions of the complex plane. All of the above have been applied to various problems on graph polynomials. (We remark that Marden's  book~\cite{Mar14}  on the geometry of polynomials is an excellent reference on classical results that we rely on for the rest of this section.  See also~\cite{ASV79,Hen74,Hol03}.)

As noted in~\cite{MRB2014}, the location of the roots of graph polynomials can be indicative of various properties of the underlying graph.
Here we investigate acyclic roots, their location and moduli, and also the properties of real roots.  Our main results are:
\begin{enumerate}[(i)]
\item If $G$ is a graph of order $n \geq 4$ and $\AC(G,x)$ is of degree $3$, then $\AC(G,x)$ has one real root and two nonreal roots $s$ and $s^\prime$ (Theorem~\ref{Thm-CubicDiscriminant}). Furthermore, both $s$ and $s^\prime$ both tend to zero as $n$ goes to infinity (Theorem~\ref{roots1}).
\item The roots of an acyclic polynomial of degree $3$ are all in the left half of the complex plane (Theorem~\ref{roots2}).
\item There exist graphs with degree $3$ acyclic polynomials having real roots of arbitrarily large modulus (Theorem~\ref{roots3}).
\item The roots of $\AC(G,x)$ (of any degree)  are all real if and only if $G$ is a forest (Theorem~\ref{roots4}).
\item There are graphs $G$ of arbitrarily large order $n$ which have a real acyclic root in $[a(n),b(n)]$ and no acyclic root of modulus larger than $|a(n)|$, where $a(n) = -\frac{n^2}{2}-\frac{5n}{2}+17$ and $b(n) = -\frac{n^2}{2}-\frac{5n}{2}+18$ (Theorem~\ref{roots5}).
\item There are arbitrarily large graphs with acyclic roots which have a positive real part (Theorem~\ref{roots6}).
\end{enumerate}
We begin our investigation of acyclic roots by looking at the roots of acyclic polynomials of small degree.

\subsection{Roots of Acyclic Polynomials of Degree 3}
\label{Sec-SmallAcyclicDimension}

The only graph whose acyclic polynomial is of degree $1$ is $K_{1}$, with polynomial $1+x$, with a root at $x=-1$.
Acyclic polynomials of degree $2$ are precisely those of
$\overline{K_2}$ ({\it i.e.}, the complement of the complete graph of order 2)
and complete graphs of order $n \geq 2$, and have the form $1 + nx + \binom{n}{2}x^2$.
The acyclic polynomials of degree 2 have roots
\[ \frac{-1}{n-1} \pm \frac{\sqrt{n^2-2n}}{n(n - 1)}i.\]
The modulus of each such root is
$\sqrt{\frac{2}{n^2-n}}$,
which is decreasing, bounded by $1$ (and tends to $0$ as $n \rightarrow \infty$).

The roots so far are rather uninteresting, but the situation changes dramatically when we move to acyclic dimension 3 (see Figure~\ref{acyclicdegree3-order12}). This subsection is devoted to such an investigation. We begin by characterizing when a graph has acyclic dimension $3$.

\begin{figure} [h!]
\centering
	\includegraphics[width=0.75\linewidth]{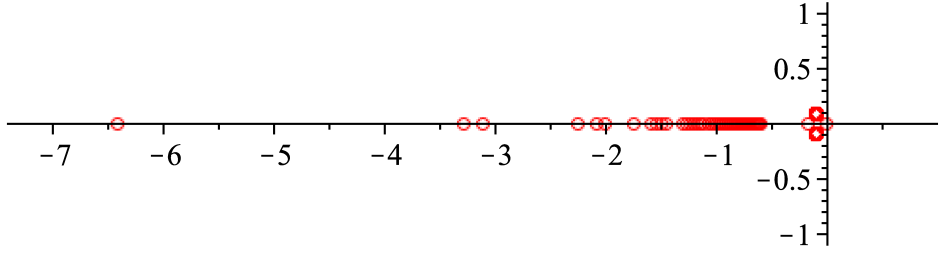}
	\caption{The acyclic roots of all graphs of order $12$ with acyclic dimension $3$.}
	\label{acyclicdegree3-order12}
\end{figure}

%
%

\begin{Theorem}\label{Degree3AcycPoly}
Let $G$ be a graph of order $n$. Then $\AC(G,x)$ has degree three if and only if one of the following holds:
\begin{enumerate}
	\item $G$ is disconnected and $G = \overline{K_3}$ or $G = K_1 \cup K_{n-1}$ with $n\geq3$.
	\item $G$ is connected and $\overline{G}$ is the disjoint union of at least two stars, at least one of which has an edge.
\end{enumerate}
\end{Theorem}

\begin{Proof}
We begin by showing the reverse direction.
First, note that the degree of the acyclic polynomial of a graph is the sum of the degrees of the acyclic polynomials of its components.

Consider the cases where $G$ is disconnected.
If $G$ is $\overline{K_3}$, then it is clear that $\AC(G,x) = (x+1)^3$, so $\AC(G,x)$ has degree three.
If $G =  K_1 \cup K_{n-1}$ with $n\geq3$, then, because $\AC(K_1,x)$ has degree one and $\AC(K_{n-1},x)$ has degree two, $\AC(G,x)$ has degree three.

Now suppose $G$ is connected and $\overline{G}$ is the disjoint union of at least two stars, at least one of which has an edge. Then $G$  has two vertices that are not adjacent and $G$ has at least three vertices.
Since any set of three vertices, two of which are not adjacent, induces an acyclic subgraph of $G$ then the degree of $\AC(G,x)$ is at least three.
It remains to show that any subset $S \subseteq V(G)$ with four vertices contains a cycle. We will show this by considering cases based on how many vertices in $S$ belong to the same component of $\overline G$.
	
Suppose $S$ contains at least three vertices from the same component of $\overline G$. Then two of these vertices, call them $u$ and $v$, must be leaves and are therefore not adjacent in $\overline G$. Furthermore, both $u$ and $v$ are joined in $\overline{G}$ to the same vertex and to no other vertex. Thus, $S$ must contain a vertex that is independent of both $u$ and $v$ in $\overline G$ ({\it i.e.}, another leaf from the same component or a vertex from another component of $\overline G$) as shown in Figure~\ref{S4UK1}. Since $S$ contains three vertices that are independent in $\overline G$, $S$ contains three vertices that form a 3-cycle $G$.

\begin{figure}[h!]
\centering
\begin{tikzpicture}
	\tikzstyle{every node}=[circle, draw, font=\small, scale=0.7]
	
	\node[scale=0.6/0.7] (b) at (-4,-3/4) {$b$};
	\node (u) at (150:2cm) {$u$};
	\node (v) at (270:2cm) {$v$};
	\node (a) at (30:2cm) {$a$};
	\node (x) at (0,0) {$x$};
	
	\draw (v)--(x)--(u);
	\draw (x)-- (a);
\end{tikzpicture}
\caption{Example of $\overline G$. If vertices $u$ and $v$ are part of $S$, then $a$ or $b$ must also be part of $S$, forming a 3-cycle in $G$.}
\label{S4UK1}
\end{figure}
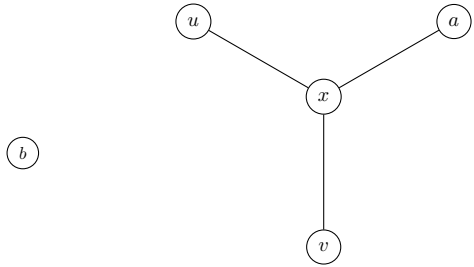

If $S$ does not contain three vertices from one component in $\overline G$, but does contain exactly two vertices, $u$ and $v$, from one component, then $S$ must contain two other vertices to which neither $u$ nor $v$ are joined in $\overline G$. Thus, these four vertices will contain a 4-cycle in $G$, as shown in Figure~\ref{C4inComp}.

\begin{figure}[h!]
\centering
\begin{tikzpicture}
	\tikzstyle{every node}=[circle, draw, font=\small, scale=0.7]
	
	\node (a) at (0,4) {$a$};
	\node[scale=0.6/0.7] (b) at (4,4) {$b$};
	\node (u) at (0,0) {$u$};
	\node (v) at (4,0) {$v$};
	
	\draw[red] (a)--(u)--(b)--(v)--(a);
	\draw[dashed] (a)--(b);
	\draw[thick, dashed] (u)--(v);
\end{tikzpicture}
\caption{If $u$ and $v$ are from the same component in $\overline G$, but $a$ and $b$ do not belong to this star, then these four vertices form a 4-cycle in $G$. The dotted lines represent edges that may or may not exist and the solid red lines show the edges that form a 4-cycle in $G$.}
\label{C4inComp}
\end{figure}
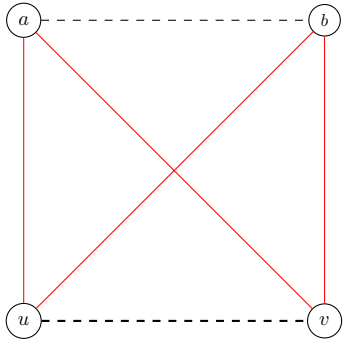

Lastly, if $S$ contains no more than one vertex from any given component of $\overline G$, then $S$ contains four vertices that are all independent in $\overline G$. These four vertices form $K_4$ in $G$, which is clearly not acyclic. Thus, $S$ is not an acyclic subset of $G$.

It follows that any subset of $V(G)$ with at least four vertices must contain a cycle. Therefore, $\AC(G,x)$ has degree three.

To show the forward direction, assume $G$ has acyclic dimension at least three (so $G$ has order  $n\geq 3$). We divide the proof into two cases based on the connectivity of $G$.

For the first case, suppose $G$ is disconnected. Recall that the acyclic dimension of $G$ equals the degree of $\AC(G,x)$,
and since the acyclic dimension of $G$ is the sum of its components' acyclic dimensions, $G$ cannot have more than three components.
If $G$ has three components, then the acyclic polynomial of each component has degree one. That is, $G = \overline{K_3}$.
Otherwise, $G$ has two components, with one component $G_1$ having an acyclic polynomial of degree one and the other $G_2$ degree two. It follows from our previous work that $G_1 = K_1$ and $G_2$ is complete of order at least $2$, that is, $G_2 = K_{n-1}$, so $G = K_1 \cup K_{n-1}$.

Now for the second case, suppose $G$ is connected. Notice that $G$ cannot contain an induced $P_4$ for otherwise $G$ would have acyclic dimension at least four. Hence $G$ is a cograph. Since $G$ is connected on $n\geq3$ vertices, $G$ must be the join of $k \geq 2$ smaller cographs, $F_1, \ldots, F_k$. Thus $\overline{G}$ is the disjoint union of $H_1 = \overline{F_1}, \ldots, H_k = \overline{F_k}$, and without loss of generality we may assume that $H_1, \ldots, H_k$ are connected. Also notice that $H_1, \ldots, H_k$ must be cographs since $P_4$ is self-complementary. Suppose -- to reach contradiction -- that some $H_i$ is not a star.
	
	If $H_i$ has a universal vertex $u$, then two vertices, $v,w$, both joined to $u$, must share an edge, as shown in Figure~\ref{S5withEdge}, as $H_i$ is not a star. So, $u,v,w$ are three independent vertices in $G$. Thus taking $u,v,w$ with any other vertex in $G$ gives an acyclic subset with four vertices ($G$ must contain another vertex as $k \geq 2$). It follows that $\AC(G,x)$ has degree greater or equal four, contradicting our assumption that $\AC(G,x)$ has degree three.

\begin{figure}[h!]
\centering
\begin{tikzpicture}
	\tikzstyle{every node}=[circle, draw, font=\small, scale=0.7]
	
	\node (u) at (0,0) {$u$};
	\node (x) at (270:2cm) {$x$};
	\node (y) at (180:2cm) {$y$};
	\node (w) at (90:2cm) {$w$};
	\node (v) at (0:2cm) {$v$};
	
	\draw (w)--(u)--(x);
	\draw (w)--(v)--(u)--(y);
\end{tikzpicture}
\caption{If $H_i$ has a universal vertex $u$, but is not a star, then two vertices not including $u$, for example $v$ and $w$, must be joined.}
\label{S5withEdge}
\end{figure}
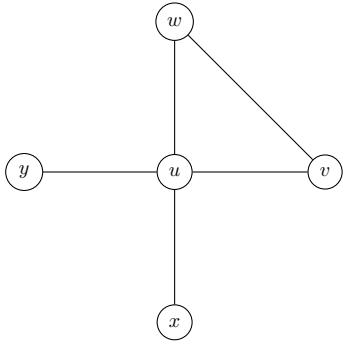

If, on the other hand, $H_i$ does not have a universal vertex then $H_i$ contains at least four vertices. Let $u$ be one of these vertices. Then $u$ must be adjacent in $H_i$ to another vertex, $v$. A third vertex, $w$, must also be adjacent in $H_i$ to one of these two vertices, say $v$. Note that $w$ cannot be adjacent to both $u$ and $v$, as otherwise $u,v,w$ will be an independent set of three vertices in $G$, which will form an acyclic subset of size four with any other vertex in $G$. Without loss of generality, $w$ is joined to $v$ (but not $u$).
	
Since $v$ is not a universal vertex, we know that there exists another vertex $x$ in $H_i$ that is not joined to $v$ in $H_i$. If $x$ is joined to either $u$ or $w$ in $\overline G$, then it is easy to see that $\{u,v,w,x\}$ is an acyclic set in $G$ (as the subgraph of $\overline G$ induced by $\{u,v,w,x\}$ properly contains a $P_4$ as a subgraph), and we have a contradiction. Thus $x$ is joined to both $u$ and $w$ (but not $v$). However, in this case $\{u,v,w,x\}$ is again an acyclic set in $G$ since the subgraph induced by
$\{u,v,w,x\}$ is $K_2 \cup K_2$ (see Figure~\ref{C4Comp}). Hence, $x$ is not adjacent to $u$, $v$, or $w$.
		
Consider a shortest path $P$ in $H_i$ from $x$ to the set $\{u,v,w\}$. Let $y$ be the last vertex on this path that is not in  $\{u,v,w\}$. By the same argument as for $x$, $y$ is not joined to $u$ or $w$ in $H_i$, so $y$ is joined to $v$ in $H_i$. In particular, $y \neq x$, so there is a previous vertex $y^\prime$ on path $P$ joined to $y$ ($y^\prime$ could be $x$). However, then in $H_i, \{u,v,y,y^\prime\}$ properly contains a $P_4$ as a subgraph, as shown in Figure~\ref{P4ExtraVertex}, and hence is acyclic in $G$, yielding a contradiction. Thus $H_i$ must be a star.
	
\begin{figure}[h!]
\centering
\begin{subfigure}[b]{0.4\linewidth}
	\centering
	\begin{tikzpicture}
		\tikzstyle{every node}=[circle, draw, font=\small, scale=0.7]
		
		\node (u) at (90:1.5cm) {$u$};
		\node (x) at (180:1.5cm) {$x$};
		\node (w) at (270:1.5cm) {$w$};
		\node (v) at (0:1.5cm) {$v$};
		
		\draw (v)--(u)--(x)--(w)--(v);
		\draw[red] (x)--(v);
		\draw[red] (w)--(u);
	\end{tikzpicture}
	\caption{If $x$ is joined to $u$ and $w$ in $H_i$, then $\{u,v,w,x\}$ is an acyclic subset of $G$. The black lines show the edges in $H_i$ ($\overline G$) and the red shows the edge in  $\overline{H_i}$ ($G$).}
	\label{C4Comp}
\end{subfigure}
\hspace{1em}
\begin{subfigure}[b]{0.4\linewidth}
	\centering
	\begin{tikzpicture}
		\tikzstyle{every node}=[circle, draw, font=\small, scale=0.7]
		
		\node (u) at (90:1.5cm) {$u$};
		\node (y) at (180:1.5cm) {$y$};
		\node (w) at (270:1.5cm) {$w$};
		\node (v) at (0:1.5cm) {$v$};
		\node[scale=0.55/0.7] (y') at (-3.5,0) {$y^\prime$};
		
		\draw[blue] (u)--(v)--(y)--(y');
		\draw (w)--(v);
	\end{tikzpicture}
	\caption{A shortest path $P$ from a vertex $x$ that is not joined to $v$ must include a (noninduced) $P_4$, shown in blue.\\ \\}
	\label{P4ExtraVertex}
\end{subfigure}
\caption{If $H_i$ is acyclic, then $H_i$ must have a universal vertex.}
\label{HiFaults}
\end{figure}
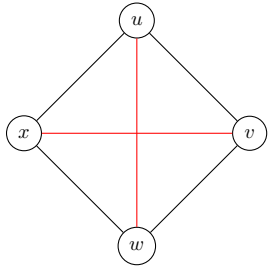
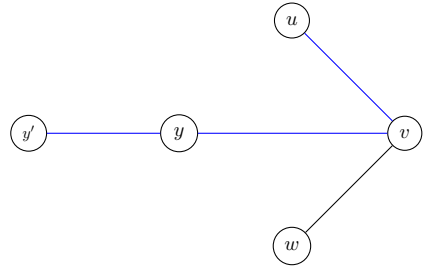

%
%

It follows that $\overline{G}$ is the disjoint union of stars. Furthermore, one of these stars must have an edge, as otherwise, $G$ will be the complete graph whose acyclic polynomial has degree two.
\end{Proof}

As it is now straightforward to recognize graphs whose acyclic polynomials have degree three, we also wish to be able to find the acyclic polynomial for any one of these graphs. We already know $a_0$, $a_1$, and $a_2$, so this task is equivalent to finding an expression for $a_3$.

\begin{Theorem} \label{a3Formula}
Let $G$ be a connected graph with order $n$ and suppose $\AC(G,x)$ has degree three so that the complement of $G$ is the disjoint union of $k$ stars $H_1,H_2,\ldots,H_k$. Let $e_i$ be the number of edges in $H_i$. Then $a_3$ is given by $\left(n-\frac{3}{2}\right)(n-k) - \frac{1}{2}\sum_{i=1}^k e_i^2$.
\end{Theorem}

\begin{Proof}
Suppose $S$ is an acyclic subset of $G$ with three vertices. Then $S$ must contain two vertices that are not joined. This corresponds to two vertices that are joined in $\overline{G}$. Since $\overline{G}$ is the disjoint union of $k$ stars, $\overline{G}$ has $n-k$ edges, {\it i.e.}, $\sum_{i=1}^k e_i = n-k$. Thus, there are $n-k$ ways to choose two vertices joined by an edge in $\overline{G}$.

Given the vertices of an edge of $\overline{G}$, there are $n-2$ ways to choose a third vertex to form an acyclic subset of $G$. However, if this third vertex belongs to the same component as the other two in $\overline G$, then we have over-counted this subset once. There are ${{e_{i}} \choose {2}} = \frac{e_i(e_i-1)}{2}$ ways to choose such a $3$-subset in $H_i$.

So, if $\overline{G}$ is the disjoint union of stars, then there are
\begin{equation*}
	(n-2)(n-k) - \sum_{i=1}^k \frac{e_i(e_i-1)}{2}
\end{equation*}
ways to form an acyclic subset with three vertices.
We then obtain
\begin{equation*}
\begin{split}
	a_3 &= (n-2)(n-k) - \sum_{i=1}^k \frac{e_i^2 - e_i}{2}\\
	& = (n-2)(n-k) - \frac{1}{2}\sum_{i=1}^k e_i^2 + \frac{n-k}{2}\\
	& = \left( n-\frac{3}{2} \right)(n-k) - \frac{1}{2}\sum_{i=1}^k e_i^2
\end{split}
\end{equation*}
\end{Proof}

Of course $a_3 \leq \binom{n}{3}$ where $n$ is the order of the graph. However, there is a much better upper bound on $a_3$ for graphs with acyclic dimension $3$. To find such a bound, we first make the following observation:

\begin{Lemma} \label{Lemma-EqualEdges}
For fixed $n$ and $k$, if $e_1,\ldots,e_k$ are nonnegative reals with $\sum_{i=1}^k e_i = n-k$, then
$$ \left( n-\frac{3}{2} \right)(n-k) - \frac{1}{2}\sum_{i=1}^k e_i^2 $$
is maximized when $e_1=e_2=\cdots=e_k = \frac{n-k}{k}$.
\end{Lemma}

\begin{Proof}
Let $\textbf{u}=(e_1, \ldots, e_k)$  and let $\textbf{v}$ be the $k$-dimensional vector $(1, \ldots, 1)$. The well-known Cauchy-Schwartz inequality states that $||\textbf{u}|| ||\textbf{v}|| \geq \textbf{u}\cdot\textbf{v}$ with equality if and only if $\textbf{u} = a\textbf{v}$ for some scalar $a$. So,
\begin{equation*}
	\sum_{i=1}^k e_i^2  = ||\textbf{u}||^2
	\geq \frac{(\textbf{u}\cdot\textbf{v})^2}{||\textbf{v}||^2}
	 = \frac{(\sum_{i=1}^k e_i)^2}{k}.
\end{equation*}
It follows that $\sum_{i=1}^k e_i^2$ is minimized when $\textbf{u}$ is a scalar multiple of $\textbf{v}$, that is, when each $e_i = \frac{n-k}{k}$. The result follows. \end{Proof}

This result allows us to find an upper bound on $a_3$ for graphs with acyclic dimension $3$ that is an order of magnitude smaller than ${n \choose 3}$.

\begin{Theorem} \label{a3UpperBound}
The leading coefficient $a_3$ of an acyclic polynomial for a graph of order $n$ with acyclic dimension three is at most
$$ f(n) = n^2 - \frac{n^2}{\sqrt{2n-2}} - \frac{n\sqrt{2n-2}}{2} - \frac{n}{2} + \frac{n}{\sqrt{2n-2}} $$
\end{Theorem}

\begin{Proof}
Let $n$ be fixed. First, if $G$ is disconnected, then either (i) $G = \overline{K_3}$ with $a_{3} =1 \leq 3/2 = f(3)$, or (ii) $G$ is the disjoint union of a single vertex and $K_{n-1}$, in which case $a_3 = \binom{n-1}{2} =  \frac{(n-1)(n-2)}{2}$. In the second case, $a_3 < f(n)$ is equivalent to
\[ \frac{(n^2+2n-2)\sqrt{2n-2}-4n^2+4n}{2\sqrt{2n-2}}>0.\]
Clearly the denominator is positive, so we need only check that the numerator is positive as well. For $n > 9$, $\sqrt{2n-2} > 4$, and the numerator is greater than
	$(n^2+2n-2) \cdot 4-4n^2+4n$, which is positive.
For $3 \leq n \leq 8$, a simple calculation will verify that the numerator is again positive. In all cases, we find that
$\frac{(n^2+2n-2)\sqrt{2n-2}-4n^2+4n}{2\sqrt{2n-2}}$
is positive so $a_3 < f(n)$.

We now assume that $G$ is connected. From Lemma~\ref{Lemma-EqualEdges}, we know that for a given $k$, $a_3$ is maximized when each $e_i = \frac{n-k}{k}$. Thus,
\begin{equation*}
\begin{split}
	a_3 &\leq \left(n-\frac{3}{2}\right)(n-k) - \frac{1}{2}\sum_{i=1}^k \left(\frac{n-k}{k}\right)^2\\
	& = \left(n-\frac{3}{2}\right)(n-k) - \frac{(n-k)^2}{2k}\\
	& = k(1 - n) - \frac{n^2}{2k} + n^2 - \frac{1}{2}n.
\end{split}
\end{equation*}
Let $B(x) = x(1 - n) - \frac{n^2}{2x} + n^2 - \frac{1}{2}n$ for $1 \leq x \leq n-1$. At the maximum value of $B$,
\begin{equation*}
	0 = \frac{dB}{dx}
	 = 1 - n + \frac{n^2}{2x^2}.
\end{equation*}
So, $x = \frac{n}{\sqrt{2n-2}}$ is a critical point of $B(x)$. The second derivative of $B$ is
\begin{equation*}
 	\frac{d^2B}{dx^2} = -\frac{n^2}{x^3} < 0.
\end{equation*}
Thus, $x = \frac{n}{\sqrt{2n-2}}$ is in fact the absolute maximum of $B$. Therefore,
\begin{equation*}
	a_3 \leq B\left( \frac{n}{\sqrt{2n-2}} \right)
	 = n^2 - \frac{n^2}{\sqrt{2n-2}} - \frac{n\sqrt{2n-2}}{2} - \frac{n}{2} + \frac{n}{\sqrt{2n-2}}
\end{equation*}
\end{Proof}

\noindent
(We remark that with more work, one can show that this upper bound is never achieved.)

We shall need as well a lower bound for $a_3$, and in this case the bound is tight.

\begin{Theorem} \label{Minimuma3}
Suppose $\AC(G,x) = a_3x^3 + \frac{n(n-1)}{2}x^2 + nx + 1$ is an acyclic polynomial with degree three. Then $a_3 \geq n - 2$, with equality if $G=K_n-e$, the complete graph of order $n \geq 3$ minus an edge.
\end{Theorem}

\begin{Proof}
Let $G$ be an arbitrary graph of order $n \geq 3$ with acyclic dimension three. Consider as well $H = K_n-e$. The complement of $H$ is the disjoint union of $K_2$ and $n-2$ independent vertices. So, following from Theorem~\ref{Degree3AcycPoly}, $\AC(H,x)$ has degree three. From Theorem~\ref{a3Formula}, the coefficient of $x^3$ in $\AC(H,x)$ is $n-2$. Now clearly $G$ is a spanning subgraph of $H$. Removing an edge does not change an acyclic subset to a cyclic subset, so any acyclic subset of $H$ is also an acyclic subset of $G$. Thus, $G$ has at least as many acyclic subsets has $H$. Therefore, $a_3 \geq n-2$.
\end{Proof}

We are now ready to explore the roots of acyclic polynomials of degree three. We know at least one root of $\AC(G,x)$ lies on the real axis and, for $n > 3$, there is only one such root.

\begin{Theorem}
\label{Thm-CubicDiscriminant}
If $G$ has $n>3$ vertices and has acyclic dimension $3$, then $\AC(G,x)$ has one real acyclic root and two nonreal acyclic roots.
\end{Theorem}

\begin{Proof}
For a general cubic, $ax^3+bx^2+cx+d$, the discriminant is defined to be $\Delta = 18abcd - 4b^3d + b^2c^2 - 4ac^3 - 27a^2d^2$. If $\Delta > 0$, then the cubic has three distinct real roots, and if $\Delta < 0$, the cubic has one real root and two nonreal roots (see, for example, \cite{irving}). We will show that the discriminant of $\AC(G,x)$ is negative, from which the result follows.

Let
\begin{equation} \label{CubicEquation}
	\AC(G,x) = ax^3+\frac{n(n-1)}{2}x^2 + nx + 1
\end{equation}
with fixed $n>3$. Then the discriminant of $\AC(G,x)$ can be calculated to be
\begin{equation*}
\begin{split}
	\Delta(a) &= 18a\frac{n(n-1)}{2}n - 4\left(\frac{n(n-1)}{2}\right)^3 + \left(\frac{n(n-1)}{2}\right)^2n^2 - 4an^3 - 27a^2\\
	& = -\frac{n^6}{4} + n^5 - \frac{5n^4}{4} + \left(5a + \frac{1}{2}\right)n^3 - 9an^2 - 27a^2.
\end{split}
\end{equation*}
To show that $\Delta(a)$ is negative, observe that its derivative with respect to $a$,
\begin{equation*}
	\Delta'(a) = 5n^3 - 9n^2 - 54a,
\end{equation*}
is negative when $a > \frac{5}{54}n^3 - \frac{1}{6}n^2$ and positive when  $a < \frac{5}{54}n^3 - \frac{1}{6}n^2$. Thus $\Delta$ is increasing to the left of $\frac{5}{54}n^3 - \frac{1}{6}n^2$ and decreasing to the right. Furthermore, at $a = \frac{5}{54}n^3 - \frac{1}{6}n^2$,
\begin{equation*}
\begin{split}
	\Delta &= -\frac{n^6}{4} + n^5 - \frac{5n^4}{4} + \left(5\left(\frac{5}{54}n^3 - \frac{1}{6}n^2\right) + \frac{1}{2}\right)n^3\\
	&\qquad - 9\left(\frac{5}{54}n^3 - \frac{1}{6}n^2\right)n^2 - 27\left(\frac{5}{54}n^3 - \frac{1}{6}n^2\right)^2\\
	& = -\frac{1}{54}n^6 + \frac{1}{6}n^5 - \frac{1}{2}n^4 + \frac{1}{2}n^3.
\end{split}
\end{equation*}
For $n\geq9$,
	$\Delta \leq -\frac{1}{6}n^5 + \frac{1}{6}n^5 - \frac{1}{2}n^4 + \frac{1}{2}n^3	 < 0$.
A quick check verifies that $\Delta$ is negative for $n = 4, \ldots,8$ as well. Since $a = \frac{5}{54}n^3 - \frac{1}{6}n^2$ is a maximum, it follows that $\Delta(a)$ is negative for all $a$ and for all $n \geq 4$. As the discriminant of $\AC(G,x)$ is negative, we conclude that $\AC(G,x)$ has one real root and two nonreal roots.
\end{Proof}

Figure~\ref{acyclicdegree3-order12} seems to suggest that all of the acyclic roots (real or otherwise) are in the left half-plane. Of course the real acyclic root of any acyclic polynomial is negative (as the polynomial has positive coefficients and hence is positive on the positive real axis), but what about the location of the two nonreal roots for acyclic polynomials of graphs of acyclic dimension three? A polynomial is said to be {\em stable} if all its roots lie in the left half-plane. To prove the stability of acyclic polynomials of degree three,  so we will use the Hermite-Biehler Theorem. To state this theorem, we first define a few terms.  A polynomial is {\em real} if each of its coefficients is real.
Given a real polynomial $f(x) = \sum b_ix^i$, then {\em even} and {\em odd} polynomials $f_e$ and $f_o$, respectively, are given by $f_e(x) = \sum b_{2i}x^i$ and $f_o(x) = \sum b_{2i+1}x^i$ (so that $f(x) = f_e(x^2)+ xf_o(x^2)$). A real polynomial is  {\em standard} if and only if it is identically $0$ or has positive leading coefficient. Finally, if
$\alpha_1 \leq \alpha_2 \leq \cdots \leq \alpha_k$ and $\beta_1 \leq \beta_2 \leq \cdots \leq \beta_\ell$ are reals, then the sequence $(\alpha_1, \alpha_2, \ldots , \alpha_k)$ {\em interlaces} the sequence $(\beta_1, \beta_2, \ldots , \beta_\ell)$ if either
\begin{enumerate}
	\item $k = \ell$ and $\alpha_1 \leq \beta_1 \leq \alpha_2 \leq \beta_2 \leq \cdots \leq \alpha_k \leq \beta_\ell$, or
	\item $k = \ell + 1$ and $\alpha_1 \leq \beta_1 \leq \alpha_2 \leq \beta_2 \leq \cdots \leq \beta_\ell \leq \alpha_k$.
\end{enumerate}

The Hermite-Biehler Theorem (see, for example, \cite{wagner}) states necessary and sufficient conditions for a real polynomial to be stable:

\begin{Theorem}[The Hermite-Biehler Theorem for Stability]
Define a standard polynomial to be a real polynomial with a positive leading coefficient (or the zero polynomial). Suppose $f(x)$ is standard.
Write $f(x)$ as
$f(x) = f_e(x^2) + xf_o(x^2)$.
Then $f(x)$ is stable if and only if the following hold
\begin{itemize}
	\item $f_e$ and $f_o$ are standard
	\item both $f_e$ and $f_o$ have all real, nonpositive roots
	\item the roots of $f_o$ interlace the roots of $f_e$.
\end{itemize}
\end{Theorem}

\begin{Theorem}\label{roots2}
If $\AC(G,x)$ is an acyclic polynomial with degree three, then $\AC(G,x)$ is stable.
\end{Theorem}

\begin{Proof}
First, if $G$ is disconnected then either (i) $G = \overline{K_3}$, $\AC(G,x) = (1+x)^3$ and the only root is $-1$, or (ii) $G = K_1 \cup K_{n-1}$, $\AC(G,x) = (1+x)(1 + (n-1)x + \frac{(n-1)(n-2)}{2}x^2)$ and the roots have real part either $-1$ and or $-1/(n-2)$. Thus in either case, the roots are all in the left half-plane, so $\AC(G,x)$ is stable.

Now we assume that $G$ is connected. Here, we will apply the \emph{Hermite-Biehler Theorem}. Since $\AC(G,x)$ has degree three, for some $a_3 \geq 1$,
\begin{equation*}
\begin{split}
	\AC(G,x) & = a_3x^3 + \frac{n(n-1)}{2}x^2 + nx + 1\\
	& = f_e(x^2) + xf_o(x^2)
\end{split}
\end{equation*}
where $f_e = \frac{n(n-1)}{2}x + 1$ and $f_o = a_3x + n$ are the \emph{even} and \emph{odd} parts of $f$, respectively. Both of these functions are standard, as is $f$.

Let $r_e$ and $r_o$ be the roots of $f_e$ and $f_o$ respectively. Then $r_e = -\frac{2}{n(n-1)}$ and $r_o=-\frac{n}{a_3}$. Notice that both $r_e$ and $r_o$ are real and negative. Furthermore, by Theorem~\ref{a3UpperBound},
\begin{equation*}
\begin{split}
	\frac{n(n-1)}{2} - \frac{a_3}{n} & \geq \frac{n(n-1)}{2} - \frac{n^2 - \frac{n^2}{\sqrt{2n-2}} - \frac{n\sqrt{2n-2}}{2} - \frac{n}{2} + \frac{n}{\sqrt{2n-2}}}{n}\\
	& = \frac{n^2}{2} - \frac{n}{2} - n + \frac{n}{\sqrt{2n-2}} + \frac{\sqrt{2n-2}}{2} + \frac{1}{2} - \frac{1}{\sqrt{2n-2}}.
\end{split}
\end{equation*}
Since $n \geq 3$, it is clear that $\frac{n}{\sqrt{2n-2}} > \frac{1}{\sqrt{2n-2}}$. Hence,
\begin{equation*}
\begin{split}
	\frac{n(n-1)}{2} - \frac{a_3}{n} & > \frac{n^2}{2} - \frac{3n}{2} + \frac{\sqrt{2n-2}}{2} + \frac{1}{2}\\
	& > \frac{n^2}{2} - \frac{3n}{2}\\
	& \geq 0.
\end{split}
\end{equation*}
So $r_o < r_e$ and thus the roots of $f_o$ interlace the roots of $f_e$. Therefore, by the Hermite-Biehler Theorem for stability, $\AC(G,x)$ is stable.
\end{Proof}

What more can we say about the location of the acyclic roots of graphs of acyclic dimension three? From Figure~\ref{acyclicdegree3-order12}, we see that, in addition to being in the left half-plane, there are real acyclic roots far to the left, but the nonreal roots seem to be close to the origin. We shall make both of these observations more precise.

We begin with an observation about the roots of certain cubics that we can apply to acyclic polynomials.

\begin{Lemma} \label{RealRootMonotonic}
Let $a_3,a_3',a_2,a_1$, and $a_0$ be positive. Suppose
\begin{equation*}
\begin{split}
	f(x) & = a_3x^3 + a_2x^2 + a_1x + a_0\\
	g(x) & = a_3'x^3 + a_2x^2 + a_1x + a_0
\end{split}
\end{equation*}
have unique real roots $r$ and $r'$ respectively. If $a_3>a_3'$, then $r>r'$.
\end{Lemma}

\begin{Proof}
Note that $r$ and $r'$ must be negative, and
\begin{equation*}
\begin{split}
	f(x) & = a_3x^3 + a_2x^2 + a_1x + a_0\\
	& = a_3'x^3 + a_2x^2 + a_1x + a_0 + a_3x^3 - a_3'x^3\\
	& = g(x) + (a_3-a_3')x^3
\end{split}
\end{equation*}
Since $r$ is the only real root of $f(x)$ and $a_3, \ldots, a_0$ are all positive, $r$ is the unique place where $f(x)$ changes from negative to positive. Similarly, $r'$ is the unique place where $g(x)$ changes from negative to positive. Thus,
$f(r') = g(r') + (a_3-a_3')(r')^3 = (a_3-a_3')(r')^3 < 0$,
since $a_3 > a_3'$. Therefore, $r > r'$.
\end{Proof}

We will use Lemma~\ref{RealRootMonotonic} to explain the behaviour of the nonreal roots of acyclic polynomials with degree three as the order of the graph increases.

\begin{Theorem}\label{roots1}
Suppose $n \geq 4$ and $s,\overline s$ are the nonreal roots of
$$ \AC(G,x) = a_3x^3 + \frac{n(n-1)}{2}x^2 + nx + 1, $$
where $\AC(G,x)$ is an acyclic polynomial with $a_3 \geq 1$. Then, as $n$ increases, $s$ and its conjugate $\overline s$ go to zero.
\end{Theorem}

\begin{Proof}
We can assume that $G$ is connected (as otherwise $G$ is $K_1 \cup K_{n-1}$ and its two nonreal acyclic roots are easily seen to head to $0$).
Suppose $r$ is the real root of $\AC(G,x)$ and let $r'$ be the real root of
\begin{equation*}
\begin{split}
	f(x) &= \left(n^2 - \frac{n^2}{\sqrt{2n-2}} - \frac{n\sqrt{2n-2}}{2} - \frac{n}{2} + \frac{n}{\sqrt{2n-2}}\right)x^3 + \frac{n(n-1)}{2}x^2 + nx + 1.
\end{split}
\end{equation*}
As $f(x)$ is of the same form as Equation~(\ref{CubicEquation}) and $n > 3$, by the same reasoning in Theorem~\ref{Thm-CubicDiscriminant}, $r'$ is unique.
From Theorem~\ref{a3UpperBound}, we know
\begin{equation*}
	a_3 \leq n^2 - \frac{n^2}{\sqrt{2n-2}} - \frac{n\sqrt{2n-2}}{2} - \frac{n}{2} + \frac{n}{\sqrt{2n-2}}.
\end{equation*}
Thus, from Lemma~\ref{RealRootMonotonic} it follows that $r<r'$. Furthermore, since $r'$ is unique, $f(x)$ will be negative for $x<r'$. At $x=-\frac{1}{2}$,
\begin{equation*}
\begin{split}
	f\left(-\frac{1}{2}\right) & = \left(n^2 - \frac{n^2}{\sqrt{2n-2}} - \frac{n\sqrt{2n-2}}{2} - \frac{n}{2} + \frac{n}{\sqrt{2n-2}}\right)\left(-\frac{1}{2}\right)^3\\\\
	& \qquad + \frac{n(n-1)}{2}\left(-\frac{1}{2}\right)^2 + n\left(-\frac{1}{2}\right) + 1\\
	& = \frac{n^2}{8\sqrt{2n-2}} + \frac{n\sqrt{2n-2}}{16} - \frac{9n}{16} - \frac{n}{8\sqrt{2n-2}} + 1\\
\end{split}
\end{equation*}
For $n\geq42$, it is straightforward to verify that
\begin{equation*}
	f\left(-\frac{1}{2}\right) \geq \frac{42n}{8\sqrt{2n-2}} + \frac{n\sqrt{82}}{16} - \frac{9n}{16} - \frac{n}{8\sqrt{2n-2}} + 1 > 0
\end{equation*}
Hence for $n \geq 42$, $r < r' < -\frac{1}{2}$.

Now notice that $\AC(G,x)$ can be factored as follows:
\[
\begin{split}
	\AC(G,x) & = a_3\left(x^3 + \frac{n(n-1)}{2a_3}x^2 + \frac{n}{a_3}x + \frac{1}{a_3}\right)\\
	& = a_3\left(x^3 + (-r-s-\overline s)x^2 + (rs + r\overline s + s\overline s)x - rs\overline s\right).
\end{split}	
\]
So, $-rs\overline s = \frac{1}{a_3}$. From Theorem~\ref{Minimuma3}, $a_3 \geq n -2$,
and thus $|s|  = \sqrt{\frac{1}{|r|a_3}}	 < \sqrt{\frac{2}{n-2}}$.
As $n \rightarrow \infty$, $\sqrt{\frac{2}{n-2}}\rightarrow 0$. Hence, $s$ and $\overline s \rightarrow 0$.
\end{Proof}

We can provide, depending on $n$, an annulus that contains the acyclic roots of graphs with acyclic dimension three.
To do so, we will use the well known  \emph{Enestr\"{o}m-Kakeya Theorem}~\cite{enestrom,kakeya}.

\begin{Theorem}[The Enestr\"{o}m-Kakeya Theorem]
Let $f(x) = a_nx^n + \cdots + a_1x^1 + a_0x^0$ be a real polynomial with degree greater than one.
If $r$ is a root of $f(x)$, then
\[
\min\left\{\frac{a_0}{a_1}, \frac{a_1}{a_2}, \ldots, \frac{a_{n-1}}{a_n}\right\}
< |r| < \max\left\{\frac{a_0}{a_1}, \frac{a_1}{a_2}, \ldots, \frac{a_{n-1}}{a_n}\right\}.
\]
\end{Theorem}

\begin{Lemma}\label{MinimumModulus}
Let $G$ be a graph with $n$ vertices.
If $r$ is a root of $\AC(G,x)$, then $\frac{1}{n} < |r|$.
Moreover, if $\AC(G,x)$ has degree three, then $|r| < \frac{n(n-1)}{2(n-2)}$.
\end{Lemma}

\begin{Proof}
As mentioned in the previous section, the Sperner bounds for complexes imply that
 $\frac{i}{n-i+1} \leq \frac{a_{i-1}}{a_i}$.
The term $\frac{i}{n-i+1}$ is at a minimum when $i=1$, so
	$\frac{a_{i-1}}{a_i} \geq \frac{i}{n-i+1} \geq \frac{1}{n}$.
From the Enestr\"{o}m-Kakeya Theorem we concude that $|r| > \frac{1}{n}.$

Now suppose
	$\AC(G,x) = 1 + nx +\frac{n(n-1)}{2}x^2 + a_3x^3$ has degree three.
From Theorem~\ref{Minimuma3}, $a_3 \geq n-2$, and so
	$\frac{a_2}{n-2} \geq \frac{a_2}{a_3}$.
Also, $n \geq 3$, and thus
\[	\frac{a_2}{n-2}  = \frac{n(n-1)}{2(n-2)}
	> \frac{2n}{n(n-1)} = \frac{a_1}{a_2}
	> \frac{1}{n} = \frac{a_0}{a_1}.
\]
By the Enestr\"{o}m-Kakeya Theorem, we conclude that $|r| < \frac{n(n-1)}{2(n-2)}$.
\end{Proof}

There are graphs whose acyclic polynomials have degree three and their real roots are within unit modulus. In particular, if $G$ is the graph of order $2n^2$ ($n \geq 2$) whose complement is the disjoint union of $n$ stars with $2n$ vertices, then
$\AC(G,x) = (4n^4 - 4n^3 - n^2 + n)x^3 + n^2(2n^2-1)x^2 + 2n^2x + 1$
and
$\max\left\{\frac{a_2}{a_3}, \frac{a_1}{a_2}, \frac{a_0}{a_1}\right\} < 1$. So via the Enestr\"{o}m-Kakeya Theorem, all three roots lie in the unit disk.

However, and more interestingly, we can show that there are
graphs with acyclic dimension three
and with real acyclic roots of large modulus.

\begin{Theorem}\label{roots3}
There exist graphs with acyclic dimension three that have real acyclic roots of arbitrarily large modulus.
\end{Theorem}

\begin{Proof}
Let $n \geq 4$ and let $G$ be the graph $K_n$ minus an edge. Then as shown in Theorem~\ref{Minimuma3},
$	\AC(G,x) = (n-2)x^3 + \frac{n(n-1)}{2}x^2 + nx + 1$.
Let $r$ be the unique real root of $\AC(G,x)$.

Note that at $x = -\frac{n}{2}$,
\[
\begin{split}
	A\left(G, -\frac{n}{2}\right) & = -\frac{(n-2)n^3}{8} + \frac{n^3(n-1)}{8} - \frac{n^2}{2} + 1\\
	& = \frac{n^3}{8} - \frac{n^2}{2} + 1.
\end{split}
\]
For $n \geq 4$,
$\AC\left(G, -\frac{n}{2}\right)  \geq 1 > 0$.
It follows that $r < -\frac{n}{2}$, since $A\left(G, -\frac{n}{2}\right) > 0$ and all the coefficients of $\AC(G,x)$ are positive. Therefore, as $n \rightarrow \infty$, the root $r \rightarrow -\infty$.

We remark that $\AC\left(G, -\frac{n}{2} -1\right)  = -\frac{n^3}{8}-\frac{n^2}{2}+\frac{n}{2}+3 < 0$ for $n \geq 4$, and so the real root $r$ will lie in $(-\frac{n}{2}-1,-\frac{n}{2})$.
\end{Proof}

\subsection{Real acyclic roots}
\label{Section-RealRoots}

For graphs of acyclic dimension at most three, we have seen that an acyclic polynomial with all real roots is rare -- it only happens for those graphs that are acyclic (so $\AC(G,x)=(1+x)^n$ and the $n$ roots are all $-1$).  Does this behaviour extend past acyclic dimension three? The related question of when a particular graph polynomial has all real roots appears difficult. For example, an important result on independence polynomials, due to Chudnovsky and Seymour~\cite{chudseymour}, states that the independence polynomials of  claw-free graphs
have all real roots, and yet this is far from a characterization of such graphs. There is much known on real-rooted chromatic polynomials \cite{dongbook} as well, and there are many examples of such families (such as the chromatic polynomials of chordal graphs).

Surprisingly, we can precisely characterize graphs that have all real acyclic roots. To do so, we will need a theorem that provides a necessary and sufficient condition for a  real polynomial to have all real roots. We begin with a definition.

The \textit{Sturm sequence} of a real polynomial $f$  is the sequence $f_0 = f,f_1 = f^{\prime},f_2,\ldots,f_k$ where, for $i \geq 2$, $f_i$ is the {\em negative} of the remainder when $f_{i-2}$ is divided by $f_{i-1}$, and $f_k$ is the last nonzero term in the sequence of polynomials of strictly decreasing degrees.
Sturm proved the following result (see \cite{Hen74,Mar14}).

\begin{Theorem}[Sturm's Theorem]\label{thmSturm}
Let $f$ be a polynomial with real coefficients and let $(f_0,f_1,\ldots,f_k)$ be its Sturm Sequence. Let $a<b$ be two real numbers that are not roots of $f$. Then the number of distinct roots of $f$ in $(a,b)$ is $V(a)-V(b)$ where $V(c)$ is the number of changes in sign in $(f_0(c),f_1(c),\ldots,f_k(c))$.
\end{Theorem}

The Sturm sequence $(f_0,f_1,\ldots,f_k)$ of $f$ has \textit{gaps in degree} if the degree of one term is at least $2$ lower than the preceding one; the Sturm sequence has a \textit{negative leading coefficient} if one of the terms does. A consequence of Sturm's Theorem that is also due to Sturm (again, see \cite{Hen74,Mar14}) will be very useful to us (we use the formulation stated in  \cite{BrownHickman2002}).

\begin{Corollary}\label{corSturm}
Let $f$ be a real polynomial whose degree and leading coefficient are both positive. Then $f$ has all real roots if and only if its Sturm sequence has no gaps in degree and no negative leading coefficients.
\end{Corollary}

We are ready to characterize graphs with all real acyclic roots.

\begin{Theorem}\label{roots4}
$G$ has all real acyclic roots if and only if $G$ is a forest.
\end{Theorem}
\begin{Proof}
We observe first that if $G$ is a forest of order $n$, then $\AC(G, x) = (1+x)^n$, which clearly has all real roots (namely $-1$ with multiplicity $n$). We now assume that $G$ is not a forest, so that it has a cycle. Let $g$ be the order of a smallest cycle in $G$, so $g \geq 3$. Suppose that $\AC(G, x)$ has degree $d = n - \nabla(G) < n$. Then
\[
\AC(G, x)  =  1+ {{n} \choose {1}}x + {{n} \choose {2}}x^2 + \cdots + {{n} \choose {g-1}}x^{g-1} + \left({{n} \choose {g}}  - \alpha\right)x^{g}
+ \cdots + a_{d}x^{d}
\]
where $\alpha$ is a positive integer.
It is clear that $\AC(G, x)$ has all real roots if and only if
\begin{eqnarray*}
f(x) & = & x^{n}\AC\left(G, \frac{1}{x} \right) \nonumber\\
 & = & x^{n}+ {{n} \choose {1}}x^{n-1} + {{n} \choose {2}}x^{n-2} + \cdots + {{n} \choose {g-1}}x^{n-(g-1)}  \nonumber \\
  & & {}+ \left( {{n} \choose {g}} - \alpha \right) x^{n-g} + \cdots + a_d x^{n-d}
\end{eqnarray*}
has all real roots.

We now consider the first few terms of the Sturm sequence of $f$ and show that $f = f_{0}$ does \underline{not} have all real roots.
Since $f_1 = f^\prime$, then
\begin{eqnarray*}
f_{1}  & = & nx^{n-1}+ (n-1){{n} \choose {1}}x^{n-2} + (n-2){{n} \choose {2}}x^{n-3} + \cdots +    (n-(g-1)){{n} \choose {g-1}}x^{n-g}    \\
 & & {}+   (n-g)\left( {{n} \choose {g}} - \alpha \right) x^{n-g-1} + \cdots + (n-d) a_d x^{n-d-1}.
\end{eqnarray*}
Now
\[ f_0 = \left( \frac{x}{n}+\frac{1}{n}\right) f_1 + \left( \alpha  \frac{n-2g+1}{n} \right)x^{n-g} + \cdots, \]
so
\[ f_2 = -\left( f_0-\left( \frac{x}{n}+\frac{1}{n}\right) f_1 \right).\]
As $n-d \geq 1$, the coefficient of $x^{n-d-1}$ in $f_2$ is clearly nonzero (since $x^{n-d}$ is the smallest power of $x$ in  $f_0$ that has a nonzero coefficient), so it follows that $f_2$ is a nonzero polynomial of degree at most  $n-g \leq n-3$. This clearly implies that the Sturm sequence of $f$ will have a gap in degree, and hence a nonreal root. We conclude that the acyclic polynomial of $G$ has a nonreal root as well.
\end{Proof}

A classical theorem of Newton states that if all of the roots of a polynomial with positive real coefficients are themselves real,
then the polynomial's coefficients are log-concave and therefore also unimodal (see, for example, \cite{comtet}).
Hence it follows from Theorem~\ref{roots4} that if $G$ is a forest then the coefficients of $\AC(G,x)$ are unimodal,
although this observation is also readily apparent from the fact that if $G$ is a forest of order $n$ then $\AC(G,x) = (1+x)^n$.

Instead of focussing on real roots, we could instead ask which rational numbers arise as acyclic roots.
The answer is rather easy, in that the set of all such roots is $\{-\frac{1}{n} : n \geq 1 \}$.
The Rational Root Theorem shows that every root of an acyclic polynomial is of the form $-\frac{1}{k}$ for some positive integer $k$ as acyclic polynomials have constant term 1 and have positive integer coefficients. On the other hand,
note that $\AC(K_{n,n},x) = 2nx \big((1+x)^n-nx-1 \big) + n^2x^2 +2(1+x)^n-1$
and hence $\AC(K_{n,n}, -\frac{1}{n}) = 0$.

\subsection{Acyclic roots of large modulus}

The real acyclic roots of large modulus that we discovered in Section~\ref{Sec-SmallAcyclicDimension} had modulus $\Omega(n)$, but from calculations it seems that this is far from the true magnitude for graphs in general. To find (real) acyclic roots of larger modulus, we will examine
the complement of the disjoint union of the star graph $S_{n-4}$ of order $n-4$ and the cycle $C_4$.
We denote this graph as $J_n=\overline{S_{n-4} \cup C_4}$, which we observe is the same as $\overline{S_{n-4}} + \overline{C_{4}}$;
see Figure~\ref{C4barJoinS6bar} for an example of such a graph.
This family of graphs may seem to be an arbitrary choice of a family to examine, but for graphs with order between $5$ and $8$,
the acyclic polynomial of $\overline{S_{n-4} \cup C_4}$ has the left-most real root, and the root of largest modulus.

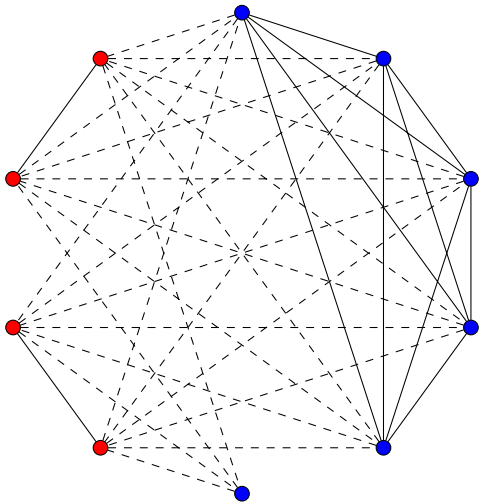
\begin{figure}[h!]
\centering
\begin{tikzpicture}[scale=0.8]
	\tikzstyle{every node}=[circle, draw, font=\small, scale=0.5]
	
	\node[fill=blue] (a) at (270:4cm) {};
	\node[fill=blue] (b) at (306:4cm) {};
	\node[fill=blue] (c) at (342:4cm) {};
	\node[fill=blue] (d) at (18:4cm) {};
	\node[fill=blue] (e) at (54:4cm) {};
	\node[fill=blue] (f) at (90:4cm) {};
	\node[fill=red] (g) at (126:4cm) {};
	\node[fill=red] (h) at (162:4cm) {};
	\node[fill=red] (i) at (198:4cm) {};
	\node[fill=red] (j) at (234:4cm) {};
	
	\draw (b)--(c)--(d)--(e)--(f)--(b)--(d)--(f)--(c)--(e)--(b);
	\draw (g)--(h);
	\draw (i)--(j);
	
	\draw[dashed] (g)--(a)--(h)--(b)--(g)--(c)--(h)--(d)--(g)--(e)--(h)--(f)--(g);
	\draw[dashed] (i)--(a)--(j)--(b)--(i)--(c)--(j)--(d)--(i)--(e)--(j)--(f)--(i);
\end{tikzpicture}
\caption{The graph $J_{10}$ is the join of $\overline{C_4}$ (in red) and $\overline{S_6}$ (in blue). The dotted lines show the join.}
\label{C4barJoinS6bar}
\end{figure}

\begin{Question}
Let $n \geq 5$. Is it the case that the acyclic polynomial of the graph
$J_n$ 
has a real root that is less than the real roots of all other acyclic polynomials of graphs of order $n$?
Moreover, does $J_n$ have the acyclic root of largest modulus for graphs of order $n$?
\end{Question}

We will show that the maximum modulus of the acyclic roots of such graphs is, in fact, quadratic in $n$. We begin by first finding the acyclic polynomial of these graphs.

\begin{Lemma}\label{StarandC4CompAcycPoly}
Let $n \geq 5$.  The acyclic polynomial of the graph $J_n$
is
\begin{equation*}
	x^4 + \frac{n^2 + 5n - 34}{2}x^3 + \frac{n^2 - n}{2}x^2 + nx + 1
\end{equation*}
\end{Lemma}

\begin{Proof}
Note that
$J_n = \overline{S_{n-4} \cup C_4} = \overline{S_{n-4}} + \overline{C_{4}} = (K_{1} \cup K_{n-5}) + 2K_{2}$.
It is easy to verify that:
\begin{eqnarray*}
A(K_{1} \cup K_{n-5},x) & = & (1+x)\left( 1 +(n-5)x + {{n-5} \choose 2}x^2 \right)\\
A(2K_{2},x) & = & (1+x)^4\\
I(K_{1} \cup K_{n-5},x) & = & (1+x)(1 +(n-5)x)\\
I(2K_{2},x) & = & (1+2x)^2.\\
\end{eqnarray*}
Applying Theorem~\ref{Thm-join}, we derive that
\begin{eqnarray*}
\AC \big( (K_{1} \cup K_{n-5}) + 2K_{2},x \big) & = &
\AC(K_{1} \cup K_{n-5},x) + \AC(2K_{2},x) \\
& &
{}+ (n-4) x I(2K_{2},x)+4 x I(K_{1} \cup K_{n-5},x)  \\
 & &
{}+ \left( {n \choose 2} - 8 (n-4) - {{n-4} \choose 2} - {{4} \choose 2}\right)x^2
- nx -1\\
 & = & x^4 + \frac{n^2 + 5n - 34}{2}x^3 + \frac{n^2 - n}{2}x^2 + nx + 1.
\end{eqnarray*}
\end{Proof}

Though we cannot show that the acyclic polynomial of $J_n$ has the minimum real root for arbitrary $n$, we can place bounds on the real root that appears to be the minimum.

\begin{Theorem}\label{roots5}
Let $n \geq 6$.  The acyclic polynomial of the graph $J_n$
has a real root between $\frac{-n^2}{2} - \frac{5n}{2} + 17$ and $\frac{-n^2}{2} - \frac{5n}{2} + 18$, and has no root of modulus larger than $\frac{n^2}{2} + \frac{5n}{2} - 17$.
\end{Theorem}

\begin{Proof}
By Lemma~\ref{StarandC4CompAcycPoly},
$\AC(J_n,x) = x^4 + \frac{n^2 + 5n - 34}{2}x^3 + \frac{n^2 - n}{2}x^2 + nx + 1$.
Let $m = \frac{-n^2}{2} - \frac{5n}{2} + 17$. Then
\begin{equation*}
	\AC(J_n,m) = \frac{1}{8}n^6 + \frac{9}{8}n^5 - \frac{53}{8}n^4 - \frac{301}{8}n^3 + \frac{369}{2}n^2 - \frac{255}{2}n + 1.
\end{equation*}
When $n \geq 8$, we have
\begin{equation*}
\begin{split}
	\AC(J_n,m) &\geq 8n^4 + 72n^3 - \frac{53}{8}n^4 - \frac{301}{8}n^3 + 1476n - \frac{255}{2}n + 1\\
	& = \frac{11}{8}n^4 + \frac{275}{8}n^3 + \frac{2697}{2}n + 1\\
	& > 0.
\end{split}
\end{equation*}
For $n=6$ and $n=7$, we can verify that $\AC(J_n,m)$ is positive as well. Hence, for $n \geq 6$, $\AC(J_n,x)$ is positive when $x = \frac{-n^2}{2} - \frac{5n}{2} + 17$.

On the other hand, when $M = \frac{-n^2}{2} - \frac{5n}{2} + 18$, we have
\begin{equation*}
	\AC(J_n,M) = -\frac{3}{4}n^5 - 3n^4 + \frac{319}{4}n^3 + 56n^2 - 2574n + 5833
\end{equation*}
and
\begin{equation*}
	\AC'(J_n,M) = -\frac{15}{4}n^4 - 12n^3 + \frac{957}{4}n^2 + 112n - 2574.
\end{equation*}
So, when $n \geq 8$,
\begin{equation*}
\begin{split}
	\AC'(J_n,M) &\leq -240n^2 - 768n + \frac{957}{4}n^2 + 112n - 2574\\
	& = -\frac{3}{4}n^2 - 656n - 2574\\
	& < 0.
\end{split}
\end{equation*}
Thus, $\AC(J_n,M)$ is decreasing when $n \geq 8$. At $n=8$, $\AC(J_n,M) = -7207$.
So, $\AC(J_n,M) \leq 0$ when $n \geq 8$. Again, we can also verify that when $n = 6$ and $n=7$, $\AC(J_n,M)$ is negative.

It follows from the Intermediate Value Theorem that for $n \geq 6$, the acyclic polynomial $\AC(J_n,x)$ has a real root $r$
such that $\frac{-n^2}{2} - \frac{5n}{2} + 17 < r < \frac{-n^2}{2} - \frac{5n}{2} + 18$.  That there is no root of modulus greater than $\frac{n^2}{2} + \frac{5n}{2} - 17$
follows from the Enestr\"{o}m-Kakeya Theorem.
\end{Proof}

\subsection{Acyclic polynomials with roots in the right half-plane}
\label{Section-AcyclicRootsRightHalfPlane}

Throughout this exposition, all the roots of acyclic polynomials that we have encountered lie in the left half-plane. This leads us to question whether there exist acyclic polynomials that have roots with a positive real component.

Instead of directly finding an acyclic polynomial with a root in the right half-plane, we present a family of graphs with limits of roots in the right half-plane.

First, we make the following definition:

\begin{Definition}
Let $G$ and $H$ be two arbitrary graphs. Then $G[H]$ is the graph formed by replacing each vertex of $G$ by a copy of graph $H$, and inserting all edges between two copies of $H$ if and only if the corresponding vertices of $G$ were adjacent (we say that we have \emph{substituted $H$ in for each vertex of $G$}; $G[H]$ is often called the \emph{lexicographic product} of $G$ and $H$).
\end{Definition}

The following lemma allows us to compute the acyclic polynomial for $G[K_{k}]$ where $G$ is a complete multipartite graph.

\begin{Lemma} \label{AcycPolyGraphComp}
Let $G$ be an arbitrary complete multipartite graph.  Then
\begin{equation*}
	\AC\left(G[K_k],x\right) = \AC(G,kx) + I\left(G,kx+\binom{k}{2}x^2\right) - I(G,kx),
\end{equation*}
where $I(G,x)$ is the independence polynomial of $G$.
\end{Lemma}

\begin{Proof}
Let $H=K_k$ and let $V_G$ and $V_H$ be the vertex sets of $G$ and $H$ respectively so that the vertex set of $G[H]$ is $V_{G[H]} = V_G \times V_H$.
Suppose $U\subseteq V_{G[H]}$ and define
$\mbox{support}(U) = \{v_i : (v_i,w)\in U \mbox{ for some } w\}$.

Suppose $(v_1,w_1),(v_2,w_2)\in U$ for some $v_1,v_2,w_1,w_2$. If $v_1,v_2$ are adjacent in $G$, then $(v_1,w_1),(v_2,w_2)$ must be adjacent in $G[H]$. Thus, if $U$ is acyclic, $\mbox{support}(U)$ must be acyclic as well.
Partition the acyclic subsets of $G$ into two sets:
\begin{enumerate}
	\item The independent subsets of $G$ (all of which are necessarily acyclic)
	\item The acyclic subsets of $G$ that are not independent.
\end{enumerate}

Suppose $\mbox{support}(U)$ belongs to the first set. Then, since no two vertices of $\mbox{support}(U)$ are adjacent, two vertices $(v_1,w_1),(v_2,w_2)\in U$ are adjacent if and only if $v_1 = v_2$ and $w_1$ is adjacent $w_2$ in $H$. Thus, since $H$ is a complete graph, $U$ is acyclic if and only if, for every $v_i\in\mbox{support}(U)$ there exist no more than two vertices in $U$ of the form $(v_i,w)$ for some $w$. Hence, the contribution of these subsets to the acyclic polynomial of $G[H]$ is
$I\left(G,kx+\binom{k}{2}x^2\right)$.

If instead $\mbox{support}(U)$ is in the second set, then $\mbox{support}(U)$ must contain two vertices that are adjacent in $G$. Because $G$ is multipartite, this means that $\mbox{support}(U)$ does not contain any isolated vertices. Thus, if there exist two vertices $(v_1,w_1),(v_1,w_1')\in U$ for some $w_1\neq w_1'$ then there must exist a third vertex $(v_2,w_2)\in U$ such that $v_2$ is adjacent $v_1$ in $G$. These three vertices must then all be joined in $G[H]$, forming a cycle. Hence, if $U$ is acyclic, then for each $v_i\in\mbox{support}(U)$ there are no two vertices in $U$ each of the form $(v_i,w)$ for some $w$.

On the other hand, if $U$ is such that for each $v_i\in\mbox{support}(U)$ there are no two vertices in $U$ each of the form $(v_i,w)$ for some $w$, then two vertices $(v_1,w_1),(v_2,w_2) \in U$ are adjacent if and only if $v_1$ is adjacent $v_2$ in $G$. Since $\mbox{support}(U)$ is acyclic, this means that $U$ must also be acyclic. That is, $U$ is acyclic if and only if for each $v_i\in\mbox{support}(U)$ there are no two vertices in $U$ each of the form $(v_i,w)$ for some $w$. Hence, the contribution of the subsets of the second type to $\AC\left(G[H],x\right)$ is
$\AC(G,kx) - I(G,kx)$.

Therefore,
$\AC\left(G[H],x\right) = I\left(G,kx+\binom{k}{2}x^2\right) + \AC(G,kx) - I(G,kx)$.
\end{Proof}

We can now apply this formula to the lexicographic product of a star graph with a complete graph.
Such are the graphs that will give us acyclic roots in the right half-plane (and indeed, we prove something stronger about the location of the limits of the roots).

First we present the Beraha-Kahane-Weiss Theorem which will be a useful tool~\cite{BKW1978}. For a family of (complex) polynomials   $\{f_n(x)\colon\ n\in\mathbb{N}\}$, we say that $z\in\mathbb{C}$ is a \textit{limit of roots} of $\{f_n(x)\colon\ n\in\mathbb{N}\}$  if there is a sequence $\{z_n \colon \ n\in\mathbb{N}\}$ such that $f_n(z_n)=0$ and $z_n\rightarrow z$ as $n\rightarrow\infty$.

\begin{Theorem}[The Beraha-Kahane-Weiss Theorem]
Suppose
$$ f_N(x) = \alpha_1(x)\lambda_1^N(x) + \cdots + \alpha_k(x)\lambda_k^N(x) $$
where the $\alpha_i$'s and the $\lambda_i$'s are polynomials in $x$
such that no $\alpha_i$ is the zero function and for no $i\neq j$ does $\lambda_i = \omega\lambda_j$ where $|\omega| = 1$. Then $z$ is a limit of roots of $f_N$ if and only if either
\begin{enumerate}
	\item for some $\ell \geq 2$, $|\lambda_{i_1}(z)| = |\lambda_{i_2}(z)| = \cdots = |\lambda_{i_\ell}(z)| > |\lambda_j(z)|$ for all $j \neq i_1, i_2, \ldots, i_\ell$, or
	\item for some $i$, $\alpha_i(z) = 0$ and for each $j \neq i$, $|\lambda_i(z)| > |\lambda_j(z)|$.
\end{enumerate}
\end{Theorem}

For the following result, we remind the reader that a {\em lima\c{c}on} is a plane curve of the form $r = a \cos \theta \pm b$ or $r = a \sin \theta \pm b$.

\begin{Theorem} \label{GraphCompRootLimits}
Suppose $G = K_{1,n-1}[K_2]$ is the lexicographic product of a star graph $K_{1,n-1}$ and the complete graph $K_2$. Then as $n \rightarrow \infty$, the roots of $\AC(G,x)$ approach a circle of radius $\frac{1}{2}$ centered at $-\frac{1}{2}$ and the truncated lima\c{c}on given by
$ r = \sqrt{2} - 2\cos\theta $
with $\frac{\pi}{4}<\theta<\frac{7\pi}{4}$ in the complex plane.
\end{Theorem}

\begin{Proof}
We have the following equations:
\begin{equation*}
\begin{split}
	\AC(K_{1,n-1},x) &= (1+x)^n\\
	I(K_{1,n-1},x) &= (1+x)^{n-1}+x.
\end{split}
\end{equation*}
Thus, from Lemma~\ref{AcycPolyGraphComp},
\begin{equation*}
\begin{split}
	\AC\left(K_{1,n-1}[K_2],x\right) &= \AC(K_{1,n-1},2x) + I(K_{1,n-1},2x+x^2) - I(K_{1,n-1},2x)\\
	& = (1+2x)^n + (1+2x+x^2)^{n-1} + 2x + x^2 - (1+2x)^{n-1} - 2x\\
	& = 2x(1+2x)^{n-1} + \left((1+x)^2\right)^{n-1} + x^2\\
	& = 2x(1+2x)^N + \left((1+x)^2\right)^N + x^2\\
\end{split}
\end{equation*}
where $N=n-1$.
Let
\begin{align*}
	\alpha_1(x) &= 2x & \lambda_1(x) &= 1+2x\\
	\alpha_2(x) &= 1 & \lambda_2(x) &= (1+x)^2\\
	\alpha_3(x) &= x^2 & \lambda_3(x) &= 1
\end{align*}
so that
$$\AC\left(K_{1,N}[K_2],x\right) = \alpha_1(x)\lambda_1^N(x) + \alpha_2(x)\lambda_2^N(x) + \alpha_3(x)\lambda_3^N(x).$$

From the Beraha-Kahane-Weiss Theorem, $z \in \mathbb{C}$ is a limit of roots of $\AC\left(K_{1,N}[K_2],x\right)$ if and only if one of the following hold:
\begin{enumerate}
\item
	\begin{enumerate}
		\item $|\lambda_1(z)|=|\lambda_2(z)|=|\lambda_3(z)|$, {\it i.e.}, $|\lambda_1(z)|=|\lambda_2(z)|=1$
		\item $|\lambda_1(z)|=|\lambda_3(z)|>|\lambda_2(z)|$, {\it i.e.}, $|\lambda_1(z)|=1>|\lambda_2(z)|$
		\item $|\lambda_2(z)|=|\lambda_3(z)|>|\lambda_1(z)|$, {\it i.e.}, $|\lambda_2(z)|=1>|\lambda_1(z)|$, or
		\item $|\lambda_1(z)|=|\lambda_2(z)|>|\lambda_3(z)|$, {\it i.e.}, $|\lambda_1(z)|=|\lambda_2(z)|>1$
	\end{enumerate}
or
\item
	\begin{enumerate}
		\item $|\lambda_1(z)|>\max\{|\lambda_2(z)|,|\lambda_3(z)|\}$, {\it i.e.}, $|\lambda_1(z)|>\max\{|\lambda_2(z)|,1\}$, and $\alpha_1(z) = 0$
		\item $|\lambda_2(z)|>\max\{|\lambda_1(z)|,|\lambda_3(z)|\}$, {\it i.e.}, $|\lambda_2(z)|>\max\{|\lambda_1(z)|,1\}$, and $\alpha_2(z) = 0$, or
		\item $|\lambda_3(z)|>\max\{|\lambda_1(z)|,|\lambda_2(z)|\}$, {\it i.e.}, $1>\max\{|\lambda_1(z)|,|\lambda_2(z)|\}$, and $\alpha_3(z) = 0$.
	\end{enumerate}
\end{enumerate}

It is easy to check that all of cases 2a, 2b and 2c lead to contradictions, so if $z$ is a limit of roots, it must satisfy one of the conditions in case 1.


Case 1a is satisfied if and only if $1 = |1+2z|$ and $1 = |1+z|^2$, {\it i.e.}, when $z$ lies on both
the circle of radius $\frac{1}{2}$ centered at $-\frac{1}{2}$
and on
the circle of radius 1 centered at $-1$.
These two circles only intersect at the origin, as shown in Figure~\ref{intersectingcircles}. Thus, $z=0$ is the only limit of roots fulfilling case 1a.

\begin{figure} [h!]
\centering
	\includegraphics[width=0.40\linewidth]{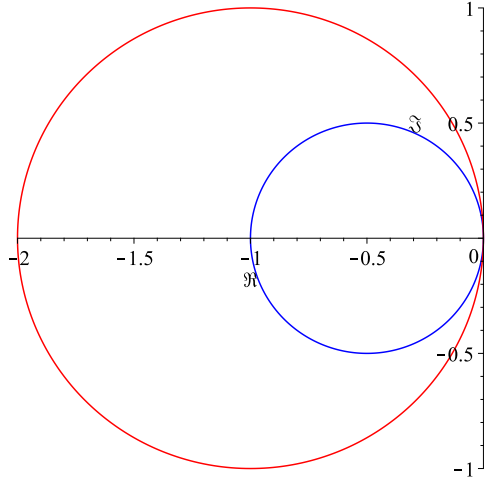}
	\caption{The circle of radius 1 centered at $-1$ (in red) and the circle of radius $\frac{1}{2}$ centered at $-\frac{1}{2}$ (in blue).}
	\label{intersectingcircles}
\end{figure}

Next, $z$ satisfies case 1b if and only if $1 = |1+2z|$ and $1 > |1+z|$. That is, when $z$ lies on the circle of radius $\frac{1}{2}$ centered at $-\frac{1}{2}$ and $z$ lies in the open circle of radius $1$ centered at $-1$. As shown in Figure~\ref{intersectingcircles}, the circle of radius $\frac{1}{2}$ centered at $-\frac{1}{2}$ is entirely contained within the circle of radius $1$ centered at $-1$, except for the point at the origin.
However, $z=0$ is also a limit of roots from case 1a. So, $z$ is a limit of roots satisfying case 1a or case 1b if and only if $z$ lies on the circle of radius $\frac{1}{2}$ centered at $- \frac{1}{2}$.

If $z$ satisfies case 1c, then $1 = |1+z|^2$ and $1 > |1+2z|$. This means that $z$ lies on the circle of radius 1 centered at $-1$ and in the open circle of radius $\frac{1}{2}$ centered at $-\frac{1}{2}$. However, the circle of radius $\frac{1}{2}$ centered at $-\frac{1}{2}$ is entirely contained within the circle of radius 1 centered at $-1$, so this is a contradiction. Hence, there is no $z$ that satisfies case 1c.

Finally, $z$ satisfies case 1d if and only if $|1+2z| = |1+z|^2$ and $|1+2z| > 1$. Since $|1+2z| > 1$, $z$ must lie outside the circle of radius $\frac{1}{2}$ centered at $-\frac{1}{2}$.
Furthermore, $|1+2z| = |1+z|^2$ if and only if $|1+2z|^2 = |1+z|^4$. Letting $z=a+bi$ with $a,b \in \mathbb{R}$, we have,
\begin{equation*}
	|1+2z|^2 = (1+2a)^2 + (2b)^2	= 1 + 4a + 4a^2 + 4b^2
\end{equation*}
and
\begin{equation*}
	|1+z|^4 = \left((1+a)^2 + b^2\right)^2
	 = 1 + 6a^2 + a^4 + b^4 + 4a  + 2b^2 + 4a^3 + 4ab^2 + 2a^2b^2.
\end{equation*}
So,
\begin{equation*}
	1 + 4a + 6a^2 + 4b^2  = 1 + 4a^2 + a^4 + b^4 + 4a  + 2b^2 + 4a^3 + 4ab^2 + 2a^2b^2.
\end{equation*}
Rearranging this expression, we get
\begin{equation*}
	2a^2 + 2b^2 = 4a^2 + a^4 + b^4 + 4a^3 + 4ab^2 + 2a^2b^2 = (a^2 + b^2 + 2a)^2.
\end{equation*}
This equation describes a lima\c{c}on. Thus, $z$ is a limit of roots satisfying case 1d if and only if $z$ lies on the lima\c{c}on in the complex plane whose equation in Cartesian form is
\begin{equation*}
	\sqrt{2}^2(a^2 + b^2) = (a^2 + b^2 + 2a)^2.
\end{equation*}
and outside the circle of radius $\frac{1}{2}$ centered at $-\frac{1}{2}$ (see Figure~\ref{limaconandcircle}).

\begin{figure}[h!]
\centering
	\includegraphics[width=0.40\linewidth]{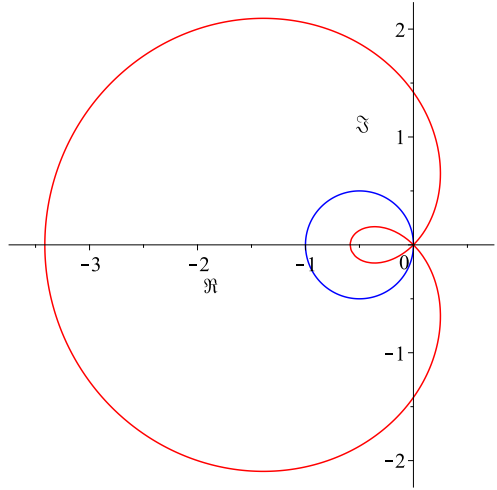}
	\caption{The lima\c{c}on given by $\sqrt{2}^2(a^2 + b^2) = (a^2 + b^2 + 2a)^2$ (in red) and the circle of radius $\frac{1}{2}$ centered at $-\frac{1}{2}$ (in blue).}
	\label{limaconandcircle}
\end{figure}

In polar form, the equation of this lima\c{c}on is
$r = \sqrt{2} - 2\cos\theta$.
So, $z$ is a limit of roots satisfying case 1d if and only if $z$ lies on the curve given by
$r = \sqrt{2} - 2\cos\theta$
with $\frac{\pi}{4}<\theta<\frac{7\pi}{4}$, shown in Figure~\ref{restrictedlimacon}.

\begin{figure}[h!]
\centering
	\includegraphics[width=0.40\linewidth]{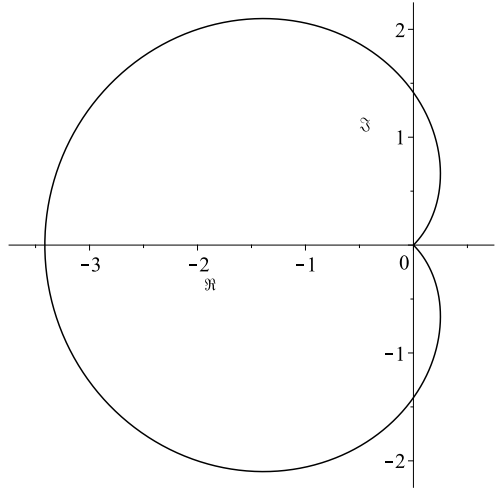}
	\caption{The curve given by $r = \sqrt{2} - 2\cos\theta$ with $\frac{\pi}{4}<\theta<\frac{7\pi}{4}$.}
	\label{restrictedlimacon}
\end{figure}

Since $z$ is a limit of roots if and only if $z$ satisfies one of cases 1a, 1b, 1c, or 1d, $z$ is a limit of roots if and only if $z$ lies on the circle of radius $\frac{1}{2}$ centered at $-\frac{1}{2}$ (as in cases 1a and 1b) or on the curve given by $r = \sqrt{2} - 2\cos\theta$ with $\frac{\pi}{4}\leq\theta\leq\frac{7\pi}{4}$ (as in cases 1a and 1d). These limits of roots are shown in Figure~\ref{limrootsSK2}.
\begin{figure}[h!]
\centering
	\includegraphics[width=0.40\linewidth]{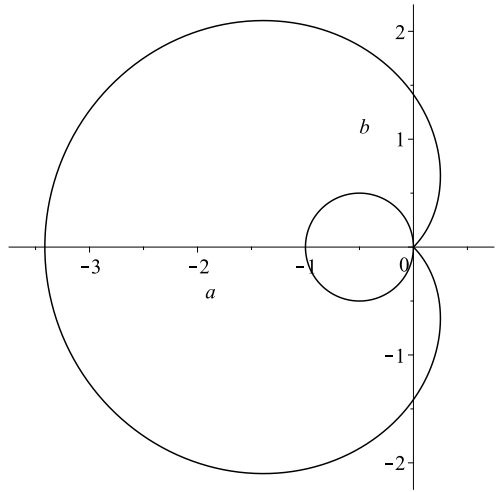}
	\caption{The limits of roots of $\AC\left(K_{1,n-1}[K_2],x\right)$.}
	\label{limrootsSK2}
\end{figure}
\end{Proof}

With this strong result, we can prove that we are guaranteed that there exist acyclic polynomials with roots in the right half-plane.

\begin{Theorem}\label{roots6}
For all sufficiently large $n$, $\AC\left(K_{1,n-1}[K_2],x\right)$ has a root with a positive real part.
\end{Theorem}

\begin{Proof}
From Theorem~\ref{GraphCompRootLimits}, we know that for large enough $n$, there exists a root of $\AC\left(K_{1,n-1}[K_2],x\right)$ arbitrarily close to the truncated lima\c{c}on given by
$ r = \sqrt{2} - 2\cos\theta $
with $\frac{\pi}{4}<\theta<\frac{7\pi}{4}$.

Take $\theta = \frac{\pi}{3}$. Then
$r  = \sqrt{2} - 2\cos\left(\frac{\pi}{3}\right)
	 = \sqrt{2} - 1$.
So $\left(\sqrt{2} - 1,\frac{\pi}{3}\right)$ is a point on this curve. In Cartesian coordinates, this means
$x  = r\cos\theta
	 = \frac{\sqrt{2} - 1}{2}$.

Hence there exist roots of $\AC\left(K_{1,n-1}[K_2],x\right)$ whose real components are arbitrarily close to $\frac{\sqrt{2} - 1}{2}$.
Since $\frac{\sqrt{2} - 1}{2} > 0$, there must exist positive roots of $\AC\left(K_{1,n-1}[K_2],x\right)$ for some $n$.
\end{Proof}

%
%
%
%

\section{Discussion and Open Problems}
\label{Section-OpenProblems}

The results of the previous section point to two types of open problems -- those that talk about the coefficients of acyclic polynomials, and those that talk about the location and nature of the acyclic roots -- and the problems are not unconnected.

In Section~\ref{Sec-SmallAcyclicDimension} we characterized those graphs having acyclic polynomials of degree $3$,
and as well we studied the coefficients and roots of their acyclic polynomials.  For higher degree we can ask:

\begin{Question}
For fixed acyclic dimension $\MF(G) \geq 4$, is there a characterization of graphs $G$ with acyclic polynomials of degree $\MF(G)$?
\end{Question}

\subsection{Unimodality}

Recall that a real polynomial $f(x) = c_0 + c_1x + c_2x^2 + \cdots$ is said to be {\em unimodal} if there exists an integer $t$ such that
$c_0 \leq c_1 \leq \cdots \leq c_{t-1} \leq c_t$ and $c_t \geq c_{t+1} \geq c_{t+2} \geq \cdots$.
Clearly $\AC(G,x)$ is unimodal in cases when $G$ is a tree or a complete graph.

Graphic polynomials (that is, generating functions for the acyclic {\em edge sets}) were recently proved to always be unimodal~\cite{Anari,BrandenHuh,lenz},
thereby settling what had been an outstanding conjecture for matroids in general for many years. 
Domination polynomials of graphs have also been conjectured to be unimodal
(see \cite{AlikhaniPeng2014} as well as \cite{BeatonBrown} for some recent progress).

For the acyclic polynomial, it is not true that $\AC(G,x)$ is unimodal for every graph $G$.
When $m \geq 20$, the acyclic polynomial $\AC(K_m + \overline{K}_6,x)$ is not unimodal because $a_3 < a_2$ and $a_3 < a_4$.
For example, observe that
$\AC(K_{20} + \overline{K}_6,x) = 1 + 26x + 325x^2 + 320x^3 + 415x^4 + 306x^5 + 121x^6 + 20x^7$.
As another example of a class of graphs with acyclic polynomials that are not unimodal,
let $H$ be complement of the join of $a$ disjoint 4-cycles, and let $G = H + \overline{K_b}$.
So $\AC(G,x) =
1
+(4a+b)x
+\binom{4a+b}{2}x^2
+\left(\binom{b}{3}+4\binom{b}{2}a+4ab+4a+4a(4a-4)\right)x^3
+\left(\binom{b}{4}+4\binom{b}{3}a+a\right)x^4
+\sum_{i=5}^b\left(\left(\binom{b}{i}+4\binom{b}{i-1}a\right)x^i\right)$.
Observe that when $a \geq 6$ and $b=8$, $\AC(G,x)$ is not unimodal (because $a_4$ is less than $a_3$ and $a_5$).
We ask:

\begin{Question}
Which classes of graphs have unimodal acyclic polynomials?  And which ones do not?
\end{Question}

Inspired by questions about unimodality of independence polynomials
of trees~\cite{AMSE1987} and of bipartite graphs~\cite{LevitMandrescu2006},
we ask:


\begin{Question}
\label{Conj-BipartiteUnimodal}
If a graph $G$ is bipartite then is $\AC(G,x)$ unimodal?
\end{Question}

As pointed out by one of the referees, the entire setting for acyclic polynomials can be embedded in a much broader setting. Suppose that ${\mathcal H}$ is a {\em hereditary} class of graphs, that is, one closed under induced subgraphs (and isomorphism). For a graph $G$ we can analogously define a generating function $P_{\mathcal H}(G,x) = \sum x^{|S|}$, where the sum is taken over all vertex subsets $S$ of $V(G)$ that induce a subgraph in  ${\mathcal H}$ (when $G$ consists of a forest, $P_{\mathcal H}(G,x) = \AC(G,x)$). A number of the results we have presented for acyclic polynomials can be carried over -- in particular, that of Theorem~\ref{roots4}, when a graph has all real acyclic roots (see \cite{MakowskyRakita}). 
In this light, it would also be interesting to find a hereditary family ${\mathcal H}$ for which $P_{\mathcal H}(G,x)$ always has unimodal coefficient sequences, but may have nonreal roots.

\subsection{Location and nature of acyclic roots}

A well known theorem due to Newton (see, for example, \cite{comtet}) states that if a polynomial with positive coefficients has all of its roots real (and on the negative real axis), then its coefficient sequence is unimodal. This result highlights the fact that the nature and location of the roots can inform the unimodality of the coefficient sequence of a real polynomial. In Section~\ref{Section-AcyclicRootsRightHalfPlane} we showed that for large enough $n$, the graph $K_{1,n-1}[K_2]$ has an acyclic root in the right half-plane.
This is not the only class of graphs with this property.
When $m \geq 18$, we find that
$\AC(K_{m} + \overline{K}_6,x) =
1+(m+6)x+\binom{m+6}{2}x^2+(15m+20)x^3+(20m+15)x^4+(15m+6)x^5+(6m+1)x^6+mx^7$
also has roots in the right half-plane.
These examples lead us to ask:

\begin{Question}
If a graph $G$ has an acyclic polynomial that is not unimodal, then does it have acyclic roots in the right half-plane?
\end{Question}

Continuing on the topic of acyclic roots, we pose the following problems:

\begin{Question}
What can be said about the nature and location of the roots of acyclic polynomials of degree $4$ and higher?
\end{Question}

\begin{Question}
Are there open sets in ${\mathbb C}$ that are free of acyclic roots? What is the closure of the acyclic roots? Is the closure of the {\em real} acyclic roots equal to $(-\infty,0]$?
\end{Question}

\begin{Question}
What is the maximum modulus of an acyclic root of a graph of order $n$?
\end{Question}

\section{Acknowledgements}
Authors J.I.\ Brown and D.A.\ Pike acknowledge research support from NSERC Discovery Grants
RGPIN-2018-05227
and RGPIN-2016-04456, respectively. The authors would also like to thank the anonymous referees for their insightful comments.
We especially thank one of the referees for pointing out that acyclic polynomials can be expressed within 
monadic second-order logic, 
for providing us with the content for Section~\ref{Sec:NotTutte},
and also for noting that Theorem~\ref{roots4} can be extended to hereditary families of graphs.




\begin{thebibliography}{99}



\bibitem{AMSE1987}Y.\ Alavi, P.J.\ Malde, A.J.\ Schwenk and P.\ Erd\H{o}s.
The vertex independence sequence of a graph is not constrained.
{\em Congressus Numerantium} 58 (1987) 15--23.

\bibitem{AB1979}M.O.\ Albertson and D.\ Berman.
A conjecture on planar graphs,
in: {\em Graph Theory and Related Topics}, eds.\ J.A.\ Bondy and U.S.R.\ Murty, Academic Press, New York, 1979, p.\ 357.

\bibitem{AlikhaniPeng2014}
S.\ Alikhani and Y.H.\ Peng.
Introduction to domination polynomial of a graph.
{\em Ars Combin.} 114 (2014) 257--266.


\bibitem{AKS1987}N.\ Alon, J.\ Kahn and P.D.\ Seymour.
Large induced degenerate subgraphs.
{\em Graphs and Combinatorics} 3 (1987) 203--211.

\bibitem{Anari}N.\ Anari, K.\ Liu, S.\ Oveis Gharan and C.\ Vinzant.
Log-concave polynomials III: Mason's ultra-log-concavity conjecture for independent sets of matroids.
arXiv:1811.01600




\bibitem{ASV79}
N.\ Anderson, E.B.\ Saff and R.S.\ Varga.
On the Enestr\"{o}m--Kakeya theorem and its sharpness.
{\em Lin.\ Alg.\ Appl.} 28 (1979) 5--16.



\bibitem{agm08}I.\ Averbouch, B.\ Godlin and J.A.\ Makowsky.
A most general edge elimination polynomial, in: \textit{International Workshop on Graph-Theoretic Concepts in Computer Science}, Springer, New York, 2008, pp.\ 31--42.

\bibitem{agm10}I.\ Averbouch, B.\ Godlin and J.A.\ Makowsky.
An extension of the bivariate chromatic polynomial. {\em Europ.\ J.\ Combin.} 31 (2010) 1--17.




\bibitem{barvinok}
A.\ Barvinok. Computing the permanent of (some) complex matrices. {\em Found.\ Comput.\ Math.} 16 (2016) 329--342.

\bibitem{BB2002}S.\ Bau and L.W.\ Beineke.
The decycling number of graphs.
{\em Australasian J.\ Combinatorics} 25 (2002) 285--298.

\bibitem{BeatonBrown}
I.\ Beaton and J.I.\ Brown.
On the unimodality of domination polynomials.
arXiv:2012.11813



\bibitem{BV1997}L.W.\ Beineke and R.C.\ Vandell.
Decycling graphs.
{\em J.\ Graph Theory} 25 (1997) 59--77.

\bibitem{BKW1978}
S.\ Beraha, J.\ Kahane and N.\ Weiss.
Limits of zeros of recursively defined families of polynomials,
in: Studies in foundations and combinatorics
(G.\ Rota, ed.) Academic Press, New York (1978) 213--232.

\bibitem{Borodin1979}
O.V.\ Borodin.
On acyclic colorings of planar graphs.
{\em Discrete Math.} 25 (1979) 211--236.

\bibitem{BrandenHuh}
P.\ Br{\"a}nd{\'e}n and J.\ Huh.
Hodge-Riemann relations for Potts model partition functions.
arXiv:1811.01696



\bibitem{brenti}
F.\ Brenti, G.F.\ Royle and D.G.\ Wagner.
Location of zeros of chromatic and related polynomials of graphs.
{\em Canad.\ J.\ Math.} 46 (1994) 55--80.

\bibitem{browncolbourn}
J.I.\ Brown and C.J.\ Colbourn.
Roots of the reliability polynomial.
{\em SIAM J.\ Discrete  Math.}  5 (1992) 571--585.

\bibitem{BrownHickman2002}J.I.\ Brown and C.A.\ Hickman.
On chromatic roots with negative real part.
{\em Ars Combin.} 63 (2002) 211--221.

\bibitem{BrownHoshino2009}J.\ Brown and R.\ Hoshino.
Independence polynomials of circulants with an application to music.
{\em Discrete Math.} 309 (2009) 2292--2304.

\bibitem{brownlucas}
J.I.\ Brown and L.\ Mol.
On the roots of all-terminal reliability polynomials.
{\em Discrete   Math.}   340 (2017) 1287--1299.

\bibitem{chudseymour}
M.\ Chudnovsky and P.\ Seymour.
The roots of the independence polynomial of a clawfree graph.
{\em J.\ Combin.\ Theory Ser.\ B} 97 (2007) 350--357.

\bibitem{colbook}
C.J. Colbourn.
\textit{The combinatorics of network reliability}, Oxford  University Press, New York, 1987.

\bibitem{comtet}
L.\ Comtet.
{\em Advanced Combinatorics},
Reidel Pub.\ Co., Boston, 1974.

\bibitem{CLSB1981}D.G.\ Corneil, H.\ Lerchs and L.\ Stewart Burlingham.
Complement reducible graphs.
{\em Disc.\ Appl.\ Math.} 3 (1981) 163--174.

\bibitem{CPS1985}D.G.\ Corneil, Y.\ Perl and L.K.\ Stewart.
A linear recognition algorithm for cographs.
{\em SIAM J.\ Comput.} 14 (1985) 926--934.


\bibitem{CEbook}
B.\ Courcelle and J.\ Engelfriet.
\textit{Graph Structure and Monadic Second-Order Logic: A Language-Theoretic Approach},
Cambridge University Press, Cambridge, 2012.



\bibitem{makowskylogic}
B.\ Courcelle, J.A.\ Makowsky and U.\ Rotics.
Linear time solvable optimization problems on graphs of bounded clique-width. {\em Th.\ Comput.\ Syst.} 33 (2000) 125--150.

\bibitem{dongbook}
F.M.\ Dong, K.M.\ Koh and K.L.\ Teo.
\textit{Chromatic Polynomials and Chromaticity Of Graphs}, World Scientific, London, 2005.

\bibitem{EMM2011}J.A.\ Ellis-Monaghan and C.\ Merino.
{Graph Polynomials and Their Applications II:  Interrelations and Interpretations},
in: Structural Analysis of Complex Networks (ed.\ M.\ Dehmer)
Springer, Dordrecht (2011) 257--292.

\bibitem{enestrom}
G.\ Enestr\"{o}m.
H\"{a}rledning af en allm\"{a}n formel f\"{o}r antalet pension\"{a}rer, som vid en godtycklig tidpunkt f\"{o}refinnas inom en sluten pensionskassa.
{\em \"{O}fversigt af Kongl. Svenska Vetenskaps-Akademien F\"{o}rhandlingar} 50 (1893) 405--415.

\bibitem{FPR1999}
P.\ Festa, P.M.\ Pardalos and M.G.C.\ Resende.
{\em Feedback set problems}, in Handbook of combinatorial optimization, Supplement Vol.\ A,
Kluwer Acad.\ Publ., Dordrecht, 1999, pp.~209--258.

\bibitem{GXWZY2015}
L.\ Gao, X.\ Xu, J.\ Wang, D.\ Zhu and Y.\ Yang.
The decycling number of generalized Petersen graphs.
{\em Disc.\ Appl.\ Math.} 181 (2015) 297--300.

\bibitem{Goodall2018}
A.\ Goodall, M.\ Hermann, T.\ Kotek, J.A.\ Makowsky and S.D.\ Noble.
On the complexity of generalized chromatic polynomials.
{\em Adv.\ Appl.\ Math.} 94 (2018) 71--102.

\bibitem{GutmanHarary1983}I.\ Gutman and F.\ Harary.
Generalizations of the matching polynomial.
{\em Utilitas Mathematica\/} 24 (1983) 97--106.

\bibitem{Hen74}
P.\ Henrici. {\em Applied and Computational Complex Analysis Vol. 1}, John Wiley and Sons, New York, 1974.

\bibitem{Hol03}
O.\ Holtz. Hermite-Biehler, Routh-Hurwitz, and total positivity. {\em Lin.\ Alg.\ Appl.} 372 (2003) 105--110.

\bibitem{irving}
R.S.\  Irving.
{\em Integers, polynomials, and rings}, Springer-Verlag, New York, 2004.

\bibitem{kakeya}
S.\ Kakeya.
On the limits of the roots of an algebraic equation with positive coefficients.
{\em Tohoku Math.\ J.} 2 (1912) 140--142.

\bibitem{Karp1972}R.M.\ Karp. Reducibility among combinatorial problems,
in: Complexity of Computer Computations (R.E.\ Miller and J.W.\ Thatcher, ed.) Plenum Press, New York-London (1972) 85--103.

\bibitem{Kogan}S.\ Kogan.
New results on large induced forests in graphs.
arXiv:1910.01356

\bibitem{lenz}
M.\ Lenz.
The $f$-vector of a representable matroid complex is log-concave.
{\em Adv.\ Appl.\ Math.}  51 (2013) 543--545.

\bibitem{LevitMandrescu2005}V.E.\ Levit and E.\ Mandrescu.
The independence polynomial of a graph -- a survey,
in: Proceedings of the 1st International Conference on Algebraic Informatics
(B.\ Bozapalidis and G.\ Rahonis, ed.) Aristotle Univ.\ Thessaloniki, Thessaloniki (2005) 233--254.

\bibitem{LevitMandrescu2006}V.E.\ Levit and E.\ Mandrescu.
Partial unimodality for independence polynomials of K\"{o}nig-Egerv\'{a}ry graphs.
{\em Congressus Numerantium} 179 (2006) 109--119.

\bibitem{LL1999}
D-M.\ Li and Y-P. Liu.
A polynomial algorithm for finding the minimum feedback vertex set of a 3-regular simple graph.
{\em Acta Math.\ Sci.} 19 (1999) 375--381.


\bibitem{Linial1986}
N.\ Linial.
Hard enumeration problems in geometry and combinatorics.
{\em SIAM J.\ Algebraic Discrete Methods} 7 (1986) 331--335.


\bibitem{MakowskyRakita}
J.A.\ Makowsky and V.\ Rakita.
Almost unimodal and real-rooted graph polynomials.
arXiv:2102.00268v2


\bibitem{MRB2014}
J.A.\ Makowsky, E.V.\ Ravve and N.K.\ Blanchard.
On the location of roots of graph polynomials.
{\em European J.\ Combinatorics} 41 (2014) 1--19.


\bibitem{Mar14}
M. Marden.
{\em Geometry of polynomials} (3rd ed.),
Amer.\ Math.\ Soc., Providence, 2014.


\bibitem{michelen}
M.\ Michelen and J.\ Sahasrabudhe.
Central limit theorems and the roots of probability generating functions.
{\em Adv.\ in Math.} 358 (2019) 106840.
(27 pages)


\bibitem{PikeZou2005}D.A.\ Pike and Y.\ Zou.
Decycling Cartesian products of two cycles.
{\em SIAM J.\ Discrete Math.} 19 (2005) 651--663.

\bibitem{roylesokal}
G.\ Royle and A.D.\ Sokal.
The Brown-Colbourn conjecture on zeros of reliability polynomials is false.
{\em J.\ Combin.\ Theory Ser.\ B}  91 (2004)  345--360.


\bibitem{ShiXu2017}L.\ Shi and H.\ Xu.
Large induced forests in graphs.
{\em J.\ Graph Theory} 85 (2017) 759--779.


\bibitem{Sperner1928}
E.\ Sperner.  Ein satz \"{u}ber untermengen einer endlichen menge.
{\em Math. Zeit.} 27 (1928) 544--548.


\bibitem{stanley}
R.P.\ Stanley.  Log-concave and unimodal sequences in algebra, combinatorics, and geometry, in: Graph Theory and Applications East and West: Proceedings of the First China-USA International Graph Theory Conference, Ann.\ New York Acad.\ Sci.\ 576 (1989) 500--535.


\bibitem{wagner}
D.G.\ Wagner.
Zeros of reliability polynomials and $f$-vectors of matroids.
{\em Combin.\ Probab.\ Comput.}  9 (2000) 167--190.

\bibitem{UKG1988}
S.\ Ueno, Y.\ Kajitani and S.\ Gotoh.
On the nonseparating independent set problem and feedback set problem for graphs with no vertex degree exceeding three.
{\em Discrete Math.} 72 (1988) 355--360.




\end{thebibliography}
\end{document}